\documentclass[11pt,reqno,oneside]{amsart}

\usepackage{graphicx}
\usepackage{amssymb}
\usepackage{epstopdf}

\usepackage[a4paper, total={6in, 9in}]{geometry}

\usepackage{amsmath,amsfonts,amsthm,mathrsfs,amssymb,cite}
\usepackage[usenames]{color}
\usepackage{subcaption}

\usepackage{xcolor}

\newtheorem{thm}{Theorem}[section]

\newtheorem{cor}[thm]{Corollary}
\newtheorem{lem}{Lemma}[section]
\newtheorem{prop}{Proposition}[section]
\theoremstyle{definition}
\newtheorem{defn}{Definition}[section]

\theoremstyle{remark}

\newtheorem{rem}{Remark}[section]
\numberwithin{equation}{section}

\numberwithin{equation}{section}

\newcounter{saveeqn}

\newcommand{\EM}[1]{\mathbf{#1}}

\usepackage{color}

\usepackage{xcolor}

\allowdisplaybreaks

\title[A novel UCP and determining conducive medium scatterers]{ On a novel UCP result and its application to inverse conductive scattering}

\author{Huaian Diao}
\address{School of Mathematics and Key Laboratory of Symbolic Computation and Knowledge Engineering of Ministry of Education, Jilin University, Changchun 130012, China}
\email{diao@jlu.edu.cn}

\author{Xiaoxu Fei}
\address{School of Mathematics and Statistics, Central South University, Changsha 410083, China\vspace*{-2mm}}
\address{and\vspace*{-2mm}}
\address{Department of Mathematics, City University of Hong Kong, Kowloon, Hong Kong SAR, China}
\email{feixx0921@163.com}

\author{Hongyu Liu}
\address{Department of Mathematics, City University of Hong Kong, Kowloon, Hong Kong SAR, China}
\email{hongyu.liuip@gmail.com, hongyliu@cityu.edu.hk}

\date{} 

\begin{document}
\maketitle

	\begin{abstract}

In this paper, we derive a novel Unique Continuation Principle (UCP) for a system of second-order elliptic PDEs and apply it to investigate inverse problems in conductive scattering. The UCP relaxes the typical assumptions imposed on the domain or boundary with certain interior transmission conditions. This is motivated by the study of the associated inverse scattering problem and enables us to establish several novel unique identifiability results for the determination of generalized conductive scatterers using a single far-field pattern, significantly extending the results in \cite{CDL20, DCL21}.
A key technical advancement in our work is the combination of Complex Geometric Optics (CGO) techniques from \cite{CDL20, DCL21} with the Fourier expansion method to microlocally analyze corner singularities and their implications for inverse problems. We believe that the methods developed can have broader applications in other contexts.


	\end{abstract}

   \section{Introduction}
  \subsection{The main results.}

  This paper addresses a novel type of the unique continuation principle (UCP) for a system of partial differential equations (PDEs) and applies this principle to investigate inverse problems in conductive medium scattering. In this subsection, we outline the main results of our study, deferring detailed proofs to subsequent sections. Initially, we introduce  the geometric setup and notations for our study.

Let $(r,\vartheta)$ denote the polar coordinates in $\mathbb{R}^2$, where $\mathbf{x}=(x_1,x_2)=(r\cos\vartheta,r\sin\vartheta)\in \mathbb{R}^2$. For a given point $\mathbf{x}_0$, we denote $B_{r_0}(\mathbf{x}_0)$ as the open disk centered at $\mathbf{x}_0$ with radius $r_0\in \mathbb{R}_+$, and we simply denote $B_{r_0}$ as $B_{r_0}(\mathbf{0})$. Let a sector $\mathcal{S}$  be defined as follows:
\begin{equation}\label{eq:s cor}
\mathcal{S}=\{\mathbf{x}\in \mathbb{R}^2~|~ \vartheta_m\leq\arg(x_1+\mathrm{i}x_2)\leq\vartheta_M\},
\end{equation}
where $-\pi \leq \vartheta_m<\vartheta_M<\pi$, $\vartheta_M-\vartheta_m\in(0,\pi)$ and $\mathrm{i}=\sqrt{-1}$. Additionally, the two boundary lines of $\mathcal{S}$ are given by
    \begin{subequations}
   \begin{align}
   	&\Gamma^{+}=\{\mathbf x\in \mathbb R^2~|~ \mathbf x=(r\cos\vartheta_M,r\sin\vartheta_M)^{\top },r>0\},\label{eq:gamma+} \\
   &\Gamma^{-}=\{\mathbf x\in \mathbb R^2~|~ \mathbf x=(r\cos\vartheta_m,r\sin\vartheta_m)^{\top },r>0\}.\label{eq:gamma-}
      \end{align}
   \end{subequations}
Furthermore, define
  \begin{equation}\label{1eq:sr0}
   \mathcal S_{r_0}=\mathcal S\cap B_{r_0},\ \Gamma^\pm_{r_0}=\Gamma^{\pm}\cap B_{r_0},\ \Lambda_{r_0}=\mathcal S\cap \partial B_{r_0}.
   \end{equation}
The opening angle of $\mathcal S_{r_0}$ is $\vartheta_M-\vartheta_m$. If $\frac{\vartheta_M-\vartheta_m}{\pi}$ is a \textit{rational number}, $\mathcal S_{r_0}$ is termed a \textit{rational corner}; otherwise, it is an \textit{irrational corner}.

Theorem \ref{thm:ucp} shows a novel type of  UCP result associated with a PDE system described by \eqref{eq:PB}. The detailed proof of Theorem \ref{thm:ucp} is given in Section \ref{sec:2}.


  \begin{thm}\label{thm:ucp}
  Let $D$ and $\Omega$  be two bounded Lipschitz domains in $\mathbb R^2$ with   connected complements $\mathbb R^2\setminus \overline{D}$ and $\mathbb R^2\setminus \overline{\Omega}$ respectively, where $\overline {\Omega} \subset D$. Assume  $\gamma_1\in \mathbb R_{+}$ and $\gamma_2\in L^\infty(D)\cap H^2(\Omega)$ with ${\rm supp}(\gamma_2-\gamma_1)\subset \Omega$ are given. Suppose  that $\Omega \cap B_{r_0}={\mathcal S}_{r_0}$, where ${\mathcal S}_{r_0}$ is defined by \eqref{1eq:sr0},  $r_0\in \mathbb R_+$ is sufficiently small and $\mathbf 0\in \partial \Omega$. The line segements $\Gamma^\pm_{r_0}$ are defined in \eqref{1eq:sr0}. Consider $\mathbf u=(v,w)^{\top}\in H^{1}(D)\times H^1(D)$ solving the PDE system:
  \begin{equation}\label{eq:PB}
  \begin{cases}
  \mathcal P\mathbf u=0\quad \mbox{in}\quad D,\\
  \mathcal B\mathbf u=0\quad \mbox{on} \quad \Gamma^\pm_{r_0},
  \end{cases}
  \end{equation}
    where
   \begin{equation}\label{op:PB}
  \mathcal P=\left (
  \begin{matrix}
  \Delta +\gamma_1&0\\
  0&\Delta+\gamma_2
  \end{matrix}
  \right),
  \quad
  \mathcal B=\left (
  \begin{matrix}
  1&-1\\
  \partial_\nu +\eta &-\partial_\nu
  \end{matrix}
  \right).
  \end{equation}
  Here  the constant $\eta\in \mathbb C$ is   defined on $\Gamma_{r_0}^\pm$ and $\partial_\nu v$ is the exterior normal derivative of $v$ with $\nu$ being the exterior normal of $\partial \Omega$.  If $\mathcal S_{r_0}$ forms an irrational corner and $\eta \neq 0$, then $\mathbf u=\mathbf 0$ in $D$.

  \end{thm}

\begin{rem}
UCP is a fundamental concept in the theory of elliptic partial differential equations (PDEs). For instance, UCP can be applied to investigate the solvability and stability of partial differential equations {\cite{ARRV,Tat,SNS}}. In Theorem \ref{thm:ucp}, we establish a novel UCP for the PDE system \eqref{eq:PB}. Specifically, if a solution $\mathbf u=(v,w)^{\top}$ to \eqref{eq:PB} exists and satisfies certain boundary conditions described by \eqref{eq:PB} on two intersecting line segments within $D$, forming an irrational convex corner, then $\mathbf u$ must be identically zero in $D$. Furthermore, the boundary conditions of $\mathbf u$ on two intersecting line segments describes certain transmission properties of $v$ and $w$ across these  segments.  This type of transmission boundary condition is particularly relevant for studying the conductive scattering problem, which will be discussed in detail in subsection \ref{subsec:1.2}.
 To the best of our knowledge, this type of UCP is particularly intriguing compared to existing literature on UCP for elliptic systems. It has important applications in the study of the inverse conductive medium scattering problem  \eqref{eq:forward}, which will be further elaborated in the subsequent subsection.
\end{rem}

\begin{rem} The PDE system described by \eqref{eq:PB} in Theorem \ref{thm:ucp} originates from the invisibility associated with the time-harmonic conductive medium scattering problem \eqref{eq:sca}. The phenomenon of invisibility in inverse scattering problems has been extensively studied (see, e.g., \cite{BL2021, BPS2014, CX, CV23, CVX23, DCL21, EH18, KSS24, PSV, PS21} and the references therein). A direct result of Theorem \ref{thm:ucp} is Corollary \ref{thm:invis}, which establishes that if a medium scatterer  contains a convex irrational corner, invisibility cannot occur. In other words, for a scatterer to be invisible while possessing a convex corner, the corner must necessarily be rational, not irrational. This conclusion aligns with the specific example provided in \cite[Example 1.5]{KSS24}, where invisible  anisotropic medium scatterers feature a rational corner with an angle of the form $\ell\pi/m$ for integers $m \geq 2$ and $1 \leq \ell < 2m - 1$. Thus, Corollary \ref{thm:invis} is nearly optimal, as demonstrated by the existence of invisible scatterers with rational corners.
	
 Moreover, the condition $\eta \neq 0$ in Theorem \ref{thm:ucp} can be interpreted as a nondegeneracy condition, which is necessary for establishing Theorem \ref{thm:ucp}. Analogous nondegeneracy conditions have been proposed to characterize invisible medium scatterers; see \cite{KSS24,CV23}. The nondegeneracy condition \cite[Equation (1.6)]{KSS24} for the invisible medium scatterer  requires a certain non-vanishing property associated with the incident wave in a neighborhood of a point on the boundary of the scatterer. This condition is used to prove that the boundary of the invisible medium scatterer is $C^{1,\alpha}$ around the underlying point. On the contrary, if the invisible medium scatterer has a corner, the nondegeneracy condition is not satisfied, as demonstrated in \cite[Examples 1.5 and 1.7]{KSS24}. Such nondegeneracy conditions were also discussed in \cite[Section 3.2]{CV23}. Therefore, the nondegeneracy condition is necessary for characterizing invisible medium scatterers in \cite{CV23,KSS24}.

\end{rem}

 We then introduce the conductive medium scattering problem \eqref{eq:sca}, which will be elaborated upon further in Subsection \ref{subsec:1.2}. Let  $\Omega$  be a conductive medium scatterer suited in a homogenous isotropic background medium $\mathbb R^2$ with the medium configurations $q$ and $\eta$, which is a bounded Lipschitz domain  with a connected complement $\mathbb R^2\setminus \overline{\Omega}$.  The medium parameter $q\in L^\infty (\Omega)$ characterises  the refractive index of the medium $\Omega$ with ${\rm supp }(q-1)\subset \Omega$ and the constant $\eta$ defined on $\partial \Omega$  signifies the conductive property of $\Omega$. In what follows, we write $(\Omega;q,\eta)$ to describe the conductive medium scatterer $\Omega$ with the physical configurations $q$ and $\eta$.

  For a given incident wave $u^i$ impinging on $\Omega$, it generates a scattered wave $u^s$. The total wave $u$ is given by $u = u^i + u^s$. The incident wave $u^i$ is a solution to the Helmholtz equation.
  The conductive medium scattering problem in $\mathbb{R}^2$ can be modeled as:
\begin{equation}\label{eq:sca}
  \begin{cases}
  \Delta u^{-} + k^2 q u^{-} = 0 \hspace{3.0cm} \text{in } \Omega, \\
  \Delta u^{+} + k^2 u^{+} = 0 \hspace{3.20cm} \text{in } \mathbb{R}^2 \setminus \overline{\Omega}, \\
  u^{+} = u^{-},\quad \partial_{\nu} u^{-} = \partial_{\nu} u^{+} + \eta u^{+} \hspace{0.65cm} \text{on } \partial \Omega, \\
  u^{+} = u^{i} + u^{s} \hspace{3.9cm} \text{in } \mathbb{R}^2 \setminus \overline{\Omega}, \\
  \lim_{r \to \infty} r^{1/2} \left( \partial_r u^s - i k u^s \right) = 0, \hspace{0.55cm} \quad r = |\mathbf{x}|,
  \end{cases}
\end{equation}
where $k \in \mathbb{R}_+$ is the wave number, and $\nu \in \mathbb{S}^1$ is the exterior unit normal vector to $\partial \Omega$. The third equation in \eqref{eq:sca} describes the conductive transmission property of the total wave $u$ across $\partial \Omega$. The final equation refers to the Sommerfeld radiation condition, characterizing the outgoing nature of the scattered wave $u^s$. Moreover, $u^{\pm}$ and $\partial_{\nu}^{\pm} u$ denote the limits of $u$ and $\partial_{\nu} u$ from the exterior and interior of $\Omega$, respectively.

The direct scattering problem \eqref{eq:sca} involves solving for the scattered wave $u^s$ given the incident wave $u^i$ and the conductive medium scatterer $(\Omega, q, \eta)$, where $\Omega$ is a bounded Lipschitz domain in $\mathbb{R}^2$ with a connected complement such that $\Omega \Subset B_R \ (R\in \mathbb R_+) $  and $u^i$ can be a solution to the Helmholtz equation $\Delta u^i+k^2u^i=0$ at least in a neighborhood of $\Omega$, including, a plane wave of the form
\begin{align}\notag
	u^i(\mathbf{x}) = e^{{\rm i} k \mathbf{x} \cdot \mathbf{d}}
\end{align}
or a Herglotz wave with kernel $g_j\in L^2(\mathbb S^{1})$ of the form
\begin{align}\notag
u^{i}(\mathbf x)=\int_{\mathbb S^1} e^{\mathrm ik \xi \cdot \mathbf x}g_j(\xi)\mathrm d \sigma(\xi),\ \xi\in \mathbb S^{1},\ \mathbf x\in \mathbb R^2
\end{align}
 or a point source of the form
 \begin{align}\notag
 	u^i(\mathbf{x}; \mathbf{z}_0) = H_0^{(1)}(k |\mathbf{x} - \mathbf{z}_0|)
 \end{align}
  associated with a wave number $k\in \mathbb R_+$.
Here, $\mathbf{d} \in \mathbb{S}^1$ is the incident direction of the plane wave, $H_0^{(1)}$ denotes the zeroth-order Hankel function of the first kind, and $\mathbf{z}_0$ signifies the location of the point source with $\mathbf{z}_0\in \mathbb R^2 \setminus \overline B_R $ with  $\Omega\Subset B_R$ and $R\in \mathbb R_+$. The function $q \in L^\infty (\Omega)$ characterizes the refractive index of the medium scatterer $\Omega$ with ${\rm supp }(q) \subset \Omega$, and the constant conductive parameter $\eta$ defined on $\partial \Omega$ signifies the conductive property of $\Omega$. On the other hand, the inverse problem corresponding to \eqref{eq:sca} aims to determine the shape of the underlying conductive medium scatterer $(\Omega, q, \eta)$ and its physical configuration $(q, \eta)$ through the measurement of the scattered wave. The novel UCP derived in Theorem \ref{thm:ucp} has several applications in the conductive medium scattering problem \eqref{eq:sca}. Specifically, in  Corollary \ref{thm:invis}, utilizing the UCP for \eqref{eq:PB}, it can be shown that the conductive medium scatterer in \eqref{eq:sca} containing a convex irrational corner cannot be invisible. The detailed proof of Corollary \ref{thm:invis} is postponed to Section \ref{sec:2}.

\begin{cor}\label{thm:invis}

Consider the conductive medium scattering problem modeled by \eqref{eq:sca}, associated with a given incident wave $u^i$ satisfying the Helmholtz equation with wave number $k$. Let the bounded Lipschitz domain $\Omega\subset \mathbb R^2$ with a connected complement in \eqref{eq:sca} describe the corresponding conductive medium scatterer $(\Omega, q, \eta)$, where ${\rm supp}(q)\subset \Omega$, $q\in L^\infty(\Omega)$ and $\eta$ is a nonzero constant. If $ \Omega$ has a convex irrational corner $\mathbf x_0$ with $q \in L^\infty(\Omega) \cap H^2(\Omega \cap B_{r_0}(\mathbf x_0 ))$, where  $r_0\in \mathbb R_+$ is sufficiently small, then $\Omega$ cannot be invisible for any wave number $k\in \mathbb R_+$.
\end{cor}



Using the unique continuation property (UCP) for \eqref{eq:PB} as established in Theorem \ref{thm:ucp}, we can address the unique identifiability for the inverse problem of the conductive medium scattering problem \eqref{eq:sca}. Specifically, in Corollary \ref{thm:uni pre}, we present both local and global unique determination of a conductive medium scatterer based on a single measurement of the scattered wave, where the single measurement refers to that taken under a fixed incident wave. The unique identifiability by a single measurement is challenging and has a long and colorful history in inverse scattering problems, often referred to as Schiffer problem in the literature; see \cite{CK18, CK} and the references therein for more details. The detailed statement of Corollary \ref{thm:uni pre} is provided in Theorem \ref{1-thm:unique}, along with the corresponding proof. Before that, we introduce the  concepts of the irrational angle, irrational polygon and convex irrational polygon.  A similar definition for irrational polygons can be founded in {\cite{CDLZ20}}.

\begin{defn}\label{1-def:polygon}

Suppose $\Omega \subset \mathbb{R}^2$ is a polygon. Let $\omega = \lambda \cdot \pi$, where $\lambda \in (0, 2)$, be an interior angle of $\Omega$. The angle $\omega$ is termed \textit{irrational} if $\lambda$ is an irrational number; otherwise, it is termed \textit{rational}. Furthermore, if $\lambda \in (0, 1)$ is an irrational number, the angle $\omega$ is referred to as a \textit{convex irrational} angle. The polygon $\Omega$ is called an \textit{irrational polygon} if all its interior angles are irrational; otherwise, it is considered rational. Additionally, the irrational polygon $\Omega$ is referred to as a \textit{convex irrational polygon} if it is convex.

\end{defn}


\begin{cor}\label{thm:uni pre}
Considering the conductive medium scattering problem \eqref{eq:sca}, if the measurements of the two scattered wave corresponding to two conductive medium scatterers, $\Omega_1$ and $\Omega_2$, under a fixed incident wave are identical, then the difference  between these two scatterers, {$ \Omega_1\aleph\Omega:=(\text{Co}(\Omega_1)\setminus \text{Co}({\Omega_2}))\cup(\text{Co}(\Omega_2)\setminus \text{Co}({\Omega_1}))$}, cannot contain a convex irrational corner, {where $\text{Co}(\Omega_i),i=1,2$ is the convex hull of $\Omega_i$}. Furthermore, if two convex irrational polygonal conductive medium scatterers, $\Omega_1$ and $\Omega_2$, produce the same scattered wave measurements for a fixed incident wave, then $\Omega_1$ and $\Omega_2$ must be identical.
\end{cor}




 When the conductive medium scatterer associated with \eqref{eq:sca} has a polygonal-cell or polygonal-nest structure (see Definitions \ref{1-def:cell} and \ref{1-def:pi-nest} for rigorous explanations),   we can uniquely identify the polygonal-nest structure based on a single measurement of the scattered wave, given some a-priori knowledge of the scatterer (see the detailed a-prior information in Definition \ref{2def:adm}). This identification leverages the novel UCP for \eqref{eq:PB}. The detailed statements of these uniqueness results can be found in Theorems \ref{2thm:un-shape} and  \ref{2thm:un-eta-nest}.

In addition to determining the shape of the conductive medium scatterer using the newly derived UCP in Theorem \ref{thm:ucp}, we also establish uniqueness results for identifying the piecewise-linear-polynomial refractive index and constant conductive parameter within the conductive polygonal-cell or polygonal-nest medium scatterer. This is achieved under some a priori information based on a single measurement of the scattered wave. These results are detailed in Theorems  \ref{2thm:un-eta-cell} and  \ref{2thm:un-eta-nest}.


\subsection{The conductive medium scattering and invisibility}\label{subsec:1.2}
The conductive medium scattering problem has important and practical applications, such as the modelling of an electromagnetic object coated with a thin layer of a highly conductive material and magnetotellurics in geophysics\cite{ak92}. Indeed, this kind of the conductive medium scattering problem can be derived by TM (transverse magnetic) polarisation from the corresponding Maxwell scattering system; see more details in \cite{Bon-Liu, CDL20}.

  The direct problem \eqref{eq:sca} aims to solve the scattered wave $u^s$ given the conductive medium scatterer $(\Omega;q,\eta)$ and the incident wave $u^i$.  For the well-posedness of the direct problem \eqref{eq:sca} it was proved that there exists a unique solution $u \in H^1_{loc}(\mathbb R^2) $ to \eqref{eq:sca} (cf. \cite{Bon,Bon-Liu}). Moreover, the scattered field $u^s$ possesses the following  asymptotic behavior
\begin{equation}\label{eq:asy}
u^s(\mathbf {\hat{x}})=\frac{e^{\mathrm ik|\mathbf x|}}{|\mathbf x|^{\frac{1}{2}}} \left( u^\infty({\mathbf {\hat{x}}}) +\mathcal O(|\mathbf x|^{-\frac{1}{2}}) \right), \ |\mathbf x|\to \infty,
\end{equation}
where $\mathbf {\hat{x}}=\frac{\mathbf x}{|\mathbf x|}$ and $u^{\infty}$ is the far-field pattern of $u^s$ associated with $u^i$. Here it can be checked that the far-field pattern $u^\infty$, which is a real analytic function defined in $\mathbb S^1$, corresponds one-to-one to $u^s$ and encodes the information of the scatterer $(\Omega;q,\eta)$. From \eqref{eq:asy} we define the forward operator by
\begin{align}\label{eq:forward}
	\mathcal F ((\Omega;q,\eta); u^i)\rightarrow u^{\infty}(\mathbf {\hat x}).
\end{align}
The inverse problem associated with \eqref{eq:forward} can be formulated as
\begin{align}\label{eq:ip}
	u^{\infty}(\mathbf {\hat x};u^i)\to (\Omega;q,\eta),
\end{align}
which intends  to identify $\Omega$ and its physical configurations $q$ and $\eta$. This type of inverse problem has practical applications in non-destructive testing and radar \cite{CK}, etc.  In this paper we will mainly focus on the uniquely determination of $(\Omega,q,\eta)$ by a single far-field measurement, namely that the far-field pattern $u^{\infty}(\mathbf {\hat x};u^i)$ is generated by a fixed incident wave $u^i$ and a fixed wave number $k$, otherwise we say that it refers to many far-field measurements.


When $u^{\infty}\equiv0$, no scattering pattern can be observed outside $\Omega$, and hence the scatterer $(\Omega,q,\eta)$ is invisible/transparent with respect to the exterior observation under the wave interrogation by $u^i$.  Then by Rellich's theorem, which indicates that $u^+=u^i$ in $\mathbb R^2\setminus \overline \Omega$, one has $(w,v)=(u^-|_{\Omega},u^i|_{\Omega})$ satisfies the following transmission eigenvalue problem:
  \begin{equation}
  \begin{cases}\label{eq:trans}
  \Delta w+k^2(1+V)w=0 \hspace*{0.9cm}\mbox{in}\ \Omega,\\
  \Delta v+k^2v=0 \hspace*{2.35cm}\mbox{in}\ \Omega,\\
  w=v,\ \partial_\nu w=\partial_\nu v+\eta v\hspace*{0.7cm}\mbox{on}\ \Gamma,
  \end{cases}
  \end{equation}
  where $V=q-1$ and $\Gamma=\partial \Omega$.

  On the other hand, given the bounded Lipschitz domain $\Omega$, $V\in L^\infty(\Omega)$  and the boundary parameter $\eta\in L^\infty(\partial\Omega )$, consider the PDE system \eqref{eq:trans}. If there exists a nontrivial pair of solutions $(w,v)$ to \eqref{eq:trans} associated with $k$, then $k$ is referred to as a transmission eigenvalue and $(w,v)$ is the corresponding pair of transmission eigenfunctions of \eqref{eq:trans}.
  Here, we would like to emphasise that throughout the article, $\eta$ can be taken to be identically zero, which reduces to the case with a normal scattering inhomogeneity of refractive index $q$, and all the results still hold equally. Hence, $(\Omega;q,\eta)$ is actually a generalized conductive medium scatterer  and the transmission eigenvalue problem \eqref{eq:trans} is generalized transmission eigenvalue problem in our study, but we shall not distinguish it. The boundary localization for an acoustic-elastic transmission eigenfunctions is revealed in \cite{Diaotang23jde}.

In the following, we shall investigate the connection between the PDE system \eqref{eq:PB} in Theorem \ref{thm:ucp} and the conductive medium scattering \eqref{eq:sca} when the invisibility of the conductive medium scatterer $(\Omega;q,\eta)$ occurs. Recall that the incident wave $u^i$ satisfying the Helmholtz equation.
Since $\Omega$ is bounded and has a connected complement, let $D:=B_R$ such that $\Omega\Subset B_R $, where $B_R$ is a disk centered at the origin with  radius $R$. As discussed earlier, when $(\Omega;q,\eta)$ is  invisible/transparent, we have $u^{\infty}\equiv0$, and by Rellich's theorem, it follows that $u^s \equiv 0$ in $D \backslash \overline{\Omega}$, which implies that $u^+\equiv u^i$ in $D \backslash \overline{\Omega} $. According to \eqref{eq:sca}, we can readily derive:
\begin{equation}\label{eq:PB new}
  \begin{cases}
  \mathcal P\mathbf u=0\quad \mbox{in}\quad D,\\
  \mathcal B\mathbf u=0\quad \mbox{on}\quad \partial \Omega,
  \end{cases}
  \end{equation}
 where $\mathbf u =(u^i,u^- \chi_{\Omega}+u^i\chi_{D\backslash \overline\Omega })\in H^1(D)\times H^1(D)$, $\mathcal P$ and $\mathcal B$ are defined in \eqref{op:PB}  with $\gamma_1=k^2$ and $\gamma_2= k^2q\chi_{\Omega}+k^2\chi_{D\backslash \overline\Omega } $. When $\Omega$ has a convex corner, meaning there are two intersecting line segments lying on  $\partial\Omega$ where the intersection point belongs to $\partial \Omega$, we can directly obtain \eqref{eq:PB} from \eqref{eq:PB new}.

\subsection{Connections to existing results and main contributions}\label{sub:3}
 The UCP is a fundamental theory with various applications in solvability, controllability, optimal stability, and inverse problems for elliptic partial differential equations (cf. \cite{CK, Horm1, Horm2}). Several forms of UCP exist for second-order elliptic partial differential equations, including the weak UCP, strong UCP, UCP for Cauchy data, and UCP across a hypersurface \cite{GL, Tat, Ler, AS}. There are multiple approaches to establishing UCP for elliptic equations, such as doubling inequalities, Carleman inequalities, and three spheres inequalities. For further discussions on this topic, we refer to \cite{Carleman,GL,GL1, Horm1, Horm2, JK85,KT01, Protter,Sogge} and references therein for a historical overview. Recently, the UCP result for the (fractional) Laplacian equation in Euclidian space or a Riemannian manifold can be found in \cite{GRSU,LLM}.
%
 For elliptic PDE systems, \cite{ST} proved that the zero set of least energy solutions for Lane-Emden systems with Neumann boundary conditions has zero measure. The weak UCP for a general second-order elliptic system was established in \cite{HLNS} under certain assumptions. Moreover, \cite{SNS} derived UCP results for elliptic PDE systems, including strong UCP, weak UCP, and UCP for local Cauchy data. The UCP for the Maxwells system can be founded in \cite{Diaozhang21ip}.


As discussed in the previous subsection, the invisibility or transparency in inverse scattering problems implies that the scatterer cannot be detected by external measurements, as the probing incident wave remains unperturbed and the corresponding scattered field is identically zero. However, when the scatterer exhibits geometrical singularities, such as corners or points of high curvature, as shown in \cite{BPS2014, CX, DFL, PSV, EH18, BL2021} for time-harmonic acoustic scattering,  invisibility or transparency cannot occur for any frequency associated with an incident plane wave. Recently, new methods based on free boundary problems have been introduced to investigate invisibility or transparency in isotropic and anisotropic acoustic medium scattering \cite{CV23, PS21, CVX23, KSS24}.

The quest to determine the geometrical and physical properties of scatterers through a single far-field measurement boasts a rich and extensive history in the realm of inverse scattering \cite{AR,BL,CK,CK18,HSV16,LZ1,Diao23book}. This pursuit, often referred to as the Schiffer problem \cite{CK18}, aims to uniquely identify the location and shape of scatterers based on a single far-field measurement. When a scatterer is impenetrable, it is termed an obstacle; otherwise, it is considered a medium. In the inverse obstacle scattering, Lax and Phillips \cite{LP} demonstrated that the shape of a sound-soft obstacle can be uniquely identified through infinite far-field measurements. However, achieving uniqueness with just a single far-field measurement has thus far required additional a-priori knowledge of the geometry, such as the size or shape of the scatterer( cf.\cite{CS83,CY03,LZ1}). The unique identifiability for  both inverse acoustic, elastic and electromagnetic impenetrable or penetrable scatterers under a-priori geometric knowledge   using a few far-field measurements or single far-field measurement   can be founded in \cite{Diaowang22jde,Diao25ip,Diaofei24ipi,Diaotang24ipi,Diaozhang21ip}.   Recently,   the first stability estimating for simultaneous determination of the acoustic polygonal  obstacle and the boundary impedance through a single far-field measurement is established in  \cite{Diaotao24ip}.
 Furthermore, the effective medium theory for a time-harmonic electromagnetic  inhomogeneous and possibly anisotropic medium with embedded obstacles can be founded in \cite{Diao25effective}.



In the context of conductive medium scattering described by \eqref{eq:sca}, uniqueness results for determining $\Omega$ based on the far-field patterns of all incident plane waves at a fixed frequency, namely infinitely many far-field measurements, have been obtained in \cite{Bon-Liu,bhk,hk20}. The unique identifiability based on a single far-field measurement for recovering a convex polygon $\Omega$ and the constant boundary parameter $\eta$ has been established in \cite{DCL21} under the generic non-vanishing assumption on the corresponding total wave field, which can be fulfilled under certain physical scenarios. When the refractive index $q$ and the conductive boundary parameter $\eta$ are piecewise constants, characterized by polygonal-nest or polygonal-cell structures, the corresponding unique determination for $q$ and $\eta$ by a single far-field pattern under the generic condition has been derived in \cite{CDL20}, whereas for the case $\eta \equiv 0$, the single far-field measurement uniquely identifiability for recovering $q$ appeared in \cite{BL}.

The main contributions of this paper can be summarized in several key points:
\begin{enumerate}
\item[a)] We derive a novel type of UCP result for an elliptic second order PDE system \eqref{eq:PB}, where the components $v$ and $w$ of the solution $\mathbf u$ to \eqref{eq:PB}  satisfy  a transmission condition across two intersecting line segments within the underlying domain. These two intersecting line segments form an irrational corner.
\item[b)] We reveal that when the conductive medium scatter $(\Omega,q,\eta)$ has a convex irrational corner $\mathbf x_0$ and the physical parameter $q$ satisfies $H^2$ regularity near $\mathbf x_0$, then $\Omega$ radiates at any wave number. Furthermore, we demonstrate that an admissible complex irrational polygonal scatterer can be uniquely determined using only a single far-field measurement, independent of its medium content.
\item[c)] We demonstrate that both the shape and medium parameters of a conductive scatterer within the polygonal-nest or polygonal-cell structure can be uniquely determined through a single far-field measurement.
\end{enumerate}

For the main contribution a), in contrast to the existing weak, strong, or Cauchy-type UCPs for elliptic systems in the literature, we consider the UCP for the case where the components of the solution to \eqref{eq:PB} fulfill a transmission boundary condition on two intersecting line segments within the domain,  forming an irrational corner. This transmission boundary condition arises from the invisibility or transparency phenomena in conductive medium scattering, which holds both practical and physical significance. We will discuss this aspect in more detail in Section \ref{sec:2}. To the best of our knowledge, this represents a novel type of UCP for second-order elliptic PDE systems, playing a crucial role in investigating invisibility and uniqueness results in inverse conductive medium scattering problems.

For the main contribution b),  we can substitute the technical condition on the a-priori information in \cite{DCL21}, which requires the non-vanishing of the total wave at the corner points, with an irrationality condition related to the interior angles of the conductive polygonal medium $\Omega$. Although ensuring the non-vanishing nature of the total wave at the corner point is feasible in generic physical scenarios, the a-priori information on the irrationality of the interior angles of the underlying scatterer is less restrictive, making the new global unique identifiability more applicable in practice.  We establish this uniqueness result by using the novel UCP result described in Theorem \ref{thm:ucp}. Furthermore, using the novel UCP established in Theorem \ref{thm:ucp}, we can show that if a conductive medium scatterer $\Omega$ has a convex irrational corner on the boundary, where the reflective index $q$ has $H^2$ regularity near the corner, then $\Omega$ radiates any incident wave $u^i$ which is a solution to the Helmholtz equation with wave number $k$.



{For the main contribution c), we establish a local uniqueness result for determining the  shape of the polygonal-nest or polygonal-cell conductive medium scatterer. This is achieved by proposing a novel type of admissible condition: the non-anishing assumption of either the total wave field or its gradient. In contrast to the previous admissibility conditions in \cite{CDL20}, which only required the non-vanishing of the total wave, the new conditions impose more requirements on the total wave. From the perspective of physical intuition, these new admissible conditions are more likely to be violated. As a result, the newly established uniqueness result is more applicable to general physical scenarios. For more in-depth discussions, please refer to Remarks \ref{rem31} and \ref{rem32}. Furthermore, we can uniquely determine the refractive index in linear polynomial form and constant conductive boundary parameter of the medium scatterer, whereas in \cite{CDL20} both the refractive index and the conductive parameter are constants. In fact, the complexity inherent in the conductive polygonal-nest and polygonal-cell medium scatterer necessitates a detailed study of the singular behavior of the solution to a coupled conductive transmission PDE problem near a corner; see Lemma \ref{2thm:conductive}. We shall proceed recursively to identify the refractive index and constant conductive boundary parameter by leveraging  the so-called complex geometrical optics (CGO) solution and the corresponding total wave field  to give a dedicated microlocal analysis near the underlying corner.}

The remainder of the paper is organized as follows. In Section \ref{sec:2}, we present the proofs of Theorems \ref{thm:ucp} and Corollaries \ref{thm:invis}-\ref{thm:uni pre}.
In Section \ref{sec:3}, we present our main results regarding the uniqueness of recovering the conductive polygonal-nest and polygonal-cell medium scatterers, including the scatterers shape, the partition of their interior polygonal-nest and polygonal-cell structure and the material parameters. Section \ref{sec:4} is devoted to the proofs of the main theorems  of Section \ref{sec:3}.

   \section{The proofs of Theorem \ref{thm:ucp} and Corollaries \ref{thm:invis}-\ref{thm:uni pre}}\label{sec:2}


In this section, we aim to provide the proofs  of Theorems \ref{thm:ucp} and Corollaries \ref{thm:invis}-\ref{thm:uni pre}. Specifically,  we present a novel UCP result for an elliptic second-order PDE system and establish both local and global uniqueness results for the shape of an admissible complex conductive polygonal medium scatterer $\Omega$, utilizing only a single far-field measurement.  To prove Theorem \ref{thm:ucp}, we first introduce some auxiliary lemmas. Drawing upon the results in \cite{CX,PSV}, we present the following lemma, which constructs a complex geometrical optics (CGO) solution.


\begin{lem}\label{1-lem:cgo}\cite{CX,PSV}
Let $\widetilde q \in H^{1,1+\varepsilon_0}(\mathbb R^2) $, where $\varepsilon_0\in (0,1)$.  Suppose $\mathbf d \in \mathbb S^1 $ and $\mathbf d^{\perp}\in \mathbb S^1$ satisfy that $\mathbf d\cdot \mathbf d^\perp=0$. Define $\rho=-\tau(\mathbf d+\mathrm i\mathbf d^\perp)$, where  $\tau\in \mathbb R_+$. If $\tau$ is sufficiently large, then there exists a complex geometrical optics (CGO) solution $u_0$ of the form
\begin{equation}\label{1-eq:cgo}
u_0(\mathbf x)=(1+\psi(\mathbf x))e^{\rho\cdot \mathbf x},
\end{equation}
which  satisfies
\begin{equation}\label{eq:cgo par}
(\Delta+k^2\widetilde q)u_0=0\mbox{ in } \mathbb R^2.
\end{equation}
Furthermore, it holds that
\begin{equation}\notag
\|\psi(\mathbf x)\|_{H^{1,8}(\mathbb R^2 )}=\mathcal O(\tau^{-\frac{2}{3}}).
\end{equation}
\end{lem}


\begin{prop}\label{1-prop:some est}\cite{DFL}
For a given $\zeta\in (0,e)$ and $s>0$, it yields that
\begin{align}\label{1-eq:some est1}
&\left|\int_{\zeta}^\infty r^{s} e^{-\mu r}\mathrm d r\right|\leq \frac{2}{\Re \mu} e^{-\frac{\zeta}{2}\Re\mu},\notag\\
&\int_{0}^{\zeta}r^{s}e^{-\mu r}\mathrm d r=\frac{\Gamma(s+1)}{\mu^{s+1}}+\mathcal O(\frac{2}{\Re \mu} e^{-\frac{\zeta}{2}\Re\mu}),
\end{align}
as $\Re{\mu}\to \infty$, where $\Gamma(s)$ stands for the Gamma function.

Moreover, let $u_0$ be the CGO solution with the parameter $\tau$ defined in \eqref{1-eq:cgo}. Suppose the corner ${\mathcal S}_{r_0}$ is defined as in \eqref{1eq:sr0}. Then, there exists a vector $\mathbf{d} \in \mathbb{S}^1$ such that
\begin{align}\label{eq:d assumption}
	0<\varsigma<\mathbf d\cdot \mathbf {\hat x}\leq 1,\quad\forall \mathbf x\in {\mathcal S}_{r_0}\backslash\{\mathbf 0\} ,\quad \mathbf {\hat x} =\mathbf x/|\mathbf x|,
\end{align}
where $\varsigma$ is a positive constant depending on ${\mathcal S}_{r_0}$ and $\mathbf{d}$. For sufficiently large $\tau$, one has
\begin{align}
&\int_{\mathcal S_{r_0}}\vert \mathbf x\vert ^s\vert u_0\vert \mathrm d\mathbf x \lesssim \tau^{-(s+\frac{29}{12})}+\tau^{-(s+2)}+\mathcal O(\tau^{-1}e^{-\frac{1}{2}\varsigma r_0\tau}),\notag\\
&\|u_0(\mathbf x)\|_{H^1(\Lambda_{r_0})}\lesssim (1+\tau)(1+\tau^{-\frac{2}{3}})e^{-\varsigma r_0 \tau},\label{1-eq:someest}\\
&\left| \int_{\Gamma_{r_0}^\pm}\vert \mathbf x\vert ^s\psi(\mathbf x)e^{\rho\cdot \mathbf x}\mathrm d\sigma\right| \lesssim \tau^{-(s+\frac{7}{6})}+\mathcal O(\tau^{-\frac{7}{6}}e^{-\frac{1}{2}\varsigma r_0\tau}),\notag
\end{align}
 where  $e^{\rho\cdot\mathbf x}\ \mbox{and} \ \psi(\mathbf x)$ are defined in \eqref{1-eq:cgo}. Throughout the rest of this paper, ``$\lesssim$" means that we only perform the leading asymptotic analysis by neglecting a generic positive constant $C$ with respect to $\tau\to \infty$, where $C$ is not a function of $\tau$. This convention is similar for ``$\gtrsim$".
\end{prop}




\begin{lem}{\cite[Section 3.4]{CK}}\label{1lem:vexpand}
Suppose that $v$ is a solution to  Helmholtz equation $(\Delta +\gamma_1)v=0$ in  $B_{r_0}$, where $\gamma_1\in \mathbb R_+$. 
Then $v$ has the following spherical wave expansion in polar coordinates around the origin:
\begin{equation}\label{1eq:vexpansion}
v=\sum_{n=0}^{\infty}(a_ne^{\mathrm in\vartheta}+b_ne^{-\mathrm in\vartheta})J_{n}(\sqrt{\gamma_1}r),\ r=\vert \mathbf x\vert,
\end{equation}
where $\mathbf x=(x_1,x_2)=r(\cos\vartheta,\sin\vartheta)\in\mathbb R^2$, $a_n$ and $b_n$ are constants and
\begin{equation}\notag
J_n(t)=\sum_{p=0}^\infty\frac{(-1)^p}{2^{n+2p}p{!}(n+p)!}t^{n+2p}
\end{equation}
is the  Bessel functions of order $n$.
\end{lem}

%
%

We are in the position to give the proof of Theorem \ref{thm:ucp}.
\begin{proof}[The proof of Theorem \ref{thm:ucp}]


According to the statement of Theorem \ref{thm:ucp}, we know that the solution $(v,w)$ to \eqref{eq:PB} satisfy the following systerm:
\begin{equation}\label{1-eq:scatter}
\begin{cases}
      \Delta w+\gamma_2 w=0\hspace*{1.91cm}  \mbox{in}\ \mathcal S_{r_0},\\
      \Delta v+\gamma_1 v=0\hspace*{2.08cm} \mbox{in}\ \mathcal S_{r_0},\\
      w=v,\ \partial_\nu w=\partial_\nu v+\eta v \hspace*{0.4cm}\mbox{on}\ \Gamma_{r_0}^{\pm}.
\end{cases}
\end{equation}
By the interior regularity of the elliptic equation, we know that $v$ is analytic in $B_{r_0}$.
When $v$ is analytic, it was proven in  \cite[Theorem 3.2]{DFL} that $v(\mathbf 0)=0$. Using \eqref{1eq:vexpansion},   one has
$$
a_0+b_0=0,
$$
which indicates that
\begin{align}
v&=\sum_{n=1}^{\infty}(a_ne^{\mathrm in\vartheta}+b_ne^{-\mathrm i n\vartheta})J_n(\sqrt{\gamma_1}r)\notag\\
&=\frac{k}{2}(a_1e^{\mathrm i\vartheta}+b_1e^{-\mathrm i\vartheta})r+\mathfrak{J_1}+\mathfrak{J_2}.\label{1-eq:vI12}
\end{align}
Here
\begin{align}
\mathfrak{J_1}&=\sum_{p=1}^\infty(a_1 e^{\mathrm i\vartheta}+b_1e^{-\mathrm i\vartheta}) \frac{(-1)^p (\sqrt{\gamma_1}r)^{2p+1}}{2^{2p+1}p!(p+1)!},\
\mathfrak{J_2}=\sum_{n=2}^\infty\sum_{p=0}^\infty(a_ne^{\mathrm in\vartheta}+b_ne^{-\mathrm i\vartheta})\frac{(-1)^p (\sqrt{\gamma_1}r)^{n+2p}}{2^{n+2p}p!(n+p)!}.\notag
\end{align}

 In what follows, we shall prove this theorem by using mathematical induction.  Namely, we shall prove that $a_\ell=b_\ell=0$ for $\ell \in \mathbb N$. Since $\gamma_2\in H^2(\mathcal S_{r_0})$, let $\widetilde {\gamma_2}$ be the Sobolev extension of $\gamma_2$ in  $H^2(\mathbb R^2)$. According to Lemma \ref{1-lem:cgo}, there exists the CGO solution $u_0$ given by \eqref{1-eq:cgo}, which satisfies
$$(\Delta+\widetilde{\gamma_2})u_0=0\ \mbox{in}\ \mathbb R^2.$$
 Hence, it yields that
\begin{equation}\label{eq:cgo par1}
(\Delta+\gamma_2)u_0=0\mbox{ in } \mathcal S_{r_0}.
\end{equation}
Denoting $u=w-v$ and using \eqref{1-eq:scatter},  one has
\begin{equation}\label{eq:u eq}
\begin{cases}
\Delta u+\gamma_2u=(\gamma_1-\gamma_2)v,\hspace*{0.75cm} \mbox{in}\ \mathcal S_{r_0},\\
u=0,\ \partial_\nu u=\eta v,\hspace*{1.8cm} \mbox{on}\ \Gamma_{r_0}^{\pm}.
\end{cases}
\end{equation}
Combing \eqref{eq:u eq} with \eqref{eq:cgo par1}, utilizing Green's formula, we derive  the following integral equation
\begin{align}
\eta\int_{\Gamma^{\pm}_{r_0}}ve^{\rho\cdot \mathbf x}\mathrm d\sigma&=-\eta\int_{\Gamma^{\pm}_{r_0}}v\psi(\mathbf x)e^{\rho \cdot \mathbf x}\mathrm d\sigma
+\int_{\mathcal S_{r_0}}(\gamma_1-\gamma_2)vu_0\mathrm d\mathbf x  \notag\\
&+\int_{\Lambda_{r_0}}\partial u_0u-\partial _\nu uu_0\mathrm d\sigma. \label{1-eq:inte}
\end{align}

Substituting \eqref{1-eq:vI12} into \eqref{1-eq:inte}, we have
\begin{align}
\frac{{\sqrt{\gamma_1}}\eta}{2}\int_{\Gamma_{r_0}^\pm}(a_1e^{\mathrm i\vartheta}+b_1e^{-\mathrm i\vartheta})re^{\rho\cdot \mathbf x}\mathrm d\sigma=\sum_{i=1}^5\mathcal G_i,\label{1-eq:int1}
\end{align}
where
\begin{align}
\mathcal G_1&=-\eta\int_{\Gamma_{r_0}^\pm}v\psi(\mathbf x)e^{\rho\cdot \mathbf x}\mathrm d\sigma,\
\mathcal G_2=\int_{\mathcal S_{r_0}}(\gamma_1-\gamma_2)vu_0\mathrm d\mathbf x,\
\mathcal G_3=\int_{\Lambda_{r_0}}\partial _{\nu}u_0u-\partial _{\nu}uu_0\mathrm d\sigma,\notag\\
\mathcal G_4&=-\eta\sum_{p=1}^\infty\frac{(-1)^p(\sqrt{\gamma_1})^{2p+1}}{2^{2p+1}p!(p+1)!}\int_{\Gamma_{r_0}^\pm}(a_1e^{\mathrm i\vartheta}+b_1e^{-\mathrm i\vartheta})r^{2p+1}e^{\rho\cdot \mathbf x}\mathrm d\sigma,\notag\\
\mathcal G_5&=-\eta\sum_{n=2}^\infty\sum_{p=0}^\infty\frac{(-1)^p(\sqrt{\gamma_1})^{n+2p}}{2^{n+2p}p!(n+p)!}\int_{\Gamma_{r_0}^\pm}(a_ne^{\mathrm in\vartheta}+b_ne^{-\mathrm in\vartheta})r^{n+2p}e^{\rho\cdot \mathbf x}\mathrm d\sigma.\notag
\end{align}

 For a given sector $\mathcal S_{r_0}$, let $ \mathbf d=(\cos\varphi,\sin\varphi)^{\top}$ be the unit vector satisfying \eqref{eq:d assumption}. Once $ \mathbf d$ is fixed, then there are two choices of $\mathbf d^\perp$ as follows,
\begin{equation}\label{1-eq:d1}
\mathbf d^{\perp}=(-\sin \varphi, \cos\varphi)^\top,
\end{equation}
and
\begin{equation}\label{1-eq:d2}
\mathbf d^\perp=
(\sin \varphi,-\cos\varphi)^{\top}.
\end{equation}

Taking $\mathbf{d}^\perp$ in the form of \eqref{1-eq:d1}, we can directly observe that the left-hand side of \eqref{1-eq:int1} satisfies:
\begin{align}
\int_{\Gamma_{r_0}^\pm}&(a_1e^{\mathrm i\vartheta}+b_1e^{-\mathrm i\vartheta})re^{\rho\cdot \mathbf x}\mathrm d\sigma
 =(a_1e^{\mathrm i\vartheta_m}+b_1e^{-\mathrm i\vartheta_m})\int_{0}^{r_0}re^{-\tau e^{\mathrm i(\vartheta_m-\varphi)r}}\mathrm dr\notag\\
&+(a_1e^{\mathrm i\vartheta_M}+b_1e^{-\mathrm i\vartheta_M})\int_{0}^{r_0}re^{-\tau e^{\mathrm i(\vartheta_M-\varphi)r}}\mathrm dr\notag\\
&=(a_1e^{\mathrm i\vartheta_m}+b_1e^{-\mathrm i\vartheta_m})\frac{\Gamma(2)}{\tau^2e^{2\mathrm i(\vartheta_m-\varphi)}}+
  (a_1e^{\mathrm i\vartheta_M}+b_1e^{-\mathrm i\vartheta_M})\frac{\Gamma(2)}{\tau^2e^{2\mathrm i(\vartheta_M-\varphi)}}\notag\\
&+\mathcal O\left(\frac{1}{\tau \cos(\vartheta_m-\varphi)}e^{-\frac{r_0}{2}\tau\cos(\vartheta_m-\varphi)}+\frac{1}{\tau \cos(\vartheta_M-\varphi)}e^{-\frac{r_0}{2}\tau\cos(\vartheta_M-\varphi)}\right).\label{1-eq:1est}
\end{align}
Without loss of generality, we assume that $\cos(\vartheta_m-\varphi)\leq \cos(\vartheta_M-\varphi)$. Hence
\begin{align*}
	&\frac{1}{\tau \cos(\vartheta_m-\varphi)}e^{-\frac{r_0}{2}\tau\cos(\vartheta_m-\varphi)}+\frac{1}{\tau \cos(\vartheta_M-\varphi)}e^{-\frac{r_0}{2}\tau\cos(\vartheta_M-\varphi)}\\
	&=\frac{1}{\tau \cos(\vartheta_m-\varphi)}e^{-\frac{r_0}{2}\lambda  \tau }\left(1+\frac{\cos(\vartheta_m-\varphi)}{\cos(\vartheta_M-\varphi)}e^{-\frac{r_0}{2}  \tau (\cos(\vartheta_M-\varphi)-\cos(\vartheta_m-\varphi))}\right),
\end{align*}
where and in what follows, $\lambda=\cos(\vartheta_m-\varphi)\in \mathbb R_+ .$  For convenience, we denote the last term in \eqref{1-eq:1est} as $\mathcal O(\frac{1}{\tau}e^{-\frac{r_0}{2}\lambda\tau})$ in the rest of the paper .


Similarly, it can be shown directly  that the left-hand side of \eqref{1-eq:int1} with $\mathbf d^{\perp}$ defined in \eqref{1-eq:d2} fulfills the following estimates
\begin{align}
\int_{\Gamma_{r_0}^\pm}&(a_1e^{\mathrm i\vartheta}+b_1e^{-\mathrm i\vartheta})re^{\rho\cdot \mathbf x}\mathrm d\sigma
 =(a_1e^{\mathrm i\vartheta_m}+b_1e^{-\mathrm i\vartheta_m})\int_{0}^{r_0}re^{-\tau e^{\mathrm i(\varphi-\vartheta_m)r}}\mathrm dr\notag\\
&+(a_1e^{\mathrm i\vartheta_M}+b_1e^{-\mathrm i\vartheta_M})\int_{0}^{r_0}re^{-\tau e^{\mathrm i(\varphi-\vartheta_M)r}}\mathrm dr\notag\\
&=(a_1e^{\mathrm i\vartheta_m}+b_1e^{-\mathrm i\vartheta_m})\frac{\Gamma(2)}{\tau^2e^{2\mathrm i(\varphi-\vartheta_m)}}+
  (a_1e^{\mathrm i\vartheta_M}+b_1e^{-\mathrm i\vartheta_M})\frac{\Gamma(2)}{\tau^2e^{2\mathrm i(\varphi-\vartheta_M)}}\notag\\
&+\mathcal O(\frac{1}{\tau}e^{-\frac{r_0}{2}\lambda\tau}).\notag
\end{align}

Next we estimate each term $\mathcal G_i$ ($i=1,\dots,5$) on the right-hand side of \eqref{1-eq:int1} with respect to $\tau$ as $\tau \rightarrow \infty$. Using  Proposition \ref{1-prop:some est}, we have
\begin{align}
\vert \mathcal G_1\vert
&\leq\left \vert \int_{\Gamma_{r_0}^\pm}v\psi e^{\rho\cdot \mathbf x}\mathrm d\sigma\right \vert \notag\\
&\leq\vert \eta\vert \sum_{n=1}^\infty\sum_{p=0}^\infty\frac{(-1)^p(\sqrt{\gamma_1})^{n+2p}(\vert a_n\vert+\vert b_n\vert) }{2^{n+2p}p!(n+p)!}\int_{\Gamma_{r_0}^\pm} r^{n+2p}\vert \psi(\mathbf x)e^{\rho\cdot \mathbf x}\vert \mathrm d\sigma\notag\\
&\lesssim \vert \eta\vert \sum_{n=1}^\infty\sum_{p=0}^\infty\frac{(-1)^p(\sqrt{\gamma_1})^{n+2p}(\vert a_n\vert+\vert b_n\vert) }{2^{n+2p}p!(n+p)!}
\left ( \tau^{-(n+2p+\frac{7}{6})}+\mathcal O(\tau^{-\frac{7}{6}}e^{-\frac{1}{2}\varsigma r_0\tau})\right)\notag\\
&\lesssim \tau^{-\frac{13}{6}}+\mathcal O(\tau^{-\frac{7}{6}}e^{-\frac{1}{2}\varsigma r_0\tau})\notag
\end{align}
and
\begin{align}
\vert \mathcal G_2\vert
&\leq \vert \gamma_1-\gamma_2\vert \int_{\mathcal S_{r_0}}\vert vu_0\vert\mathrm d\mathbf x\notag\\
&\leq \vert \gamma_1-\gamma_2\vert  \sum_{n=1}^\infty\sum_{p=0}^\infty\frac{(-1)^p(\sqrt{\gamma_1})^{n+2p}(\vert a_n\vert+\vert b_n\vert) }{2^{n+2p}p!(n+p)!} \int_{\mathcal S_{r_0}}r^{n+2p}\vert u_0\vert \mathrm d\mathbf x\notag\\
&\lesssim \tau^{-\frac{41}{12}}+\tau^{-3}+\mathcal O(\tau^{-1}e^{-\frac{1}{2}\varsigma r_0\tau}).\notag
\end{align}
Moreover, there holds that
\begin{align}
\vert \mathcal G_3\vert&\leq  \|\partial_{\nu}u_0\|_{L^2(\Lambda_{r_0})}\|u\|_{L^2(\Lambda_{r_0})}
+\|u_0\|_{H^{\frac{1}{2}}(\Lambda_{r_0})}\|\partial_{\nu}u\|_{H^{-\frac{1}{2}}(\Lambda_{r_0})}.\notag\\
&\leq C(\|\partial_\nu u_0\|_{L^2(\Lambda_{r_0})}+\|u_0\|_{H^1(\Lambda_{r_0})})\|u\|_{H^1(\Lambda_{r_0})}\notag\\
&\leq\|u_0\|_{H^1(\Lambda_{r_0})} +\|\nabla u_0\|_{L^2(\Lambda_{r_0})}\lesssim (1+\tau)(1+\tau^{-\frac{2}{3}})e^{-\varsigma r_0\tau}.\notag
\end{align}
As for the last two terms of $\mathcal G_4\ \mbox{and}\ \mathcal G_5$, similarly we can deduce that
\begin{align}\label{1-eq:G4-5}
\vert \mathcal G_4+\mathcal G_5\vert &\lesssim \tau^{-3}+\mathcal O(\tau^{-1}e^{-\frac{1}{2}\varsigma r_0\tau}).
\end{align}

Combining \eqref{1-eq:1est}$-$\eqref{1-eq:G4-5} with \eqref{1-eq:int1} and letting $\tau\to \infty$, this shows that
\begin{equation}\notag
\begin{cases}
(a_1e^{\mathrm i\vartheta_m}+b_1e^{-\mathrm i\vartheta_m})e^{2\mathrm i(\varphi-\vartheta_m)}+(a_1e^{\mathrm i\vartheta_M}+b_1e^{-\mathrm i\vartheta_M})e^{2\mathrm i(\varphi-\vartheta_M)}=0,\notag\\
(a_1e^{\mathrm i\vartheta_m}+b_1e^{-\mathrm i\vartheta_m})e^{2\mathrm i(\vartheta_m-\varphi)}+(a_1e^{\mathrm i\vartheta_M}+b_1e^{-\mathrm i\vartheta_M})e^{2\mathrm i(\vartheta_M-\varphi)}=0.
\end{cases}
\end{equation}
By multiplying $e^{-2\mathrm i\varphi}$ on both two sides of the first identity above and  $e^{2\mathrm i\varphi}$ on the second, we get that
\begin{equation}\notag
A\left(
\begin{matrix}
a_1\\
b_1
\end{matrix}
\right)=\mathbf 0,
\end{equation}
where
\begin{equation}\notag
A=\left(
\begin{matrix}
e^{-\mathrm i\vartheta_m}+ e^{-\mathrm i\vartheta_M} & e^{-3\mathrm i\vartheta_m}+e^{-3\mathrm i\vartheta_M}\\
e^{3\mathrm i\vartheta_m}+ e^{3\mathrm i\vartheta_M}& e^{\mathrm i\vartheta_m}+e^{\mathrm i\vartheta_M}
\end{matrix}
\right).
\end{equation}
Since $\vartheta_M-\vartheta_m\not=\frac{\pi}{2}$ by noting $\mathcal S_{r_0}$ is  irrational, after some calculations, it shows that
\begin{equation}\notag
\det(A)=2(\cos(\vartheta_M-\vartheta_m)-\cos(3(\vartheta_M-\vartheta_m)))\not=0,
\end{equation}
which implies that $a_1=b_1=0.$

Suppose that
\begin{align}\label{eq:ass ajbj}
	a_j=b_j=0, \quad j=1,\ldots, \ell, \quad \ell \geq 1,
\end{align}
then we shall verify that $a_{\ell+1}=b_{\ell+1}=0$. Under the assumption \eqref{eq:ass ajbj}, using \eqref{1eq:vexpansion}, we have
\begin{align}
v&=\sum_{n=\ell+1}^\infty(a_ne^{\mathrm in\vartheta}+b_ne^{-\mathrm in\vartheta})J_{n}(\sqrt{\gamma_1}r)=\frac{(\sqrt{\gamma_1})^{\ell+1}(a_{\ell+1}e^{\mathrm i(\ell+1)\vartheta}+b_{\ell+1}e^{-\mathrm i(\ell+1)\vartheta})}{2^{\ell+1}(\ell+1)!}r^{\ell+1}\notag\\
&+\sum_{p=1}^\infty\frac{(-1)^p(\sqrt{\gamma_1})^{2p+\ell+1}(a_{\ell+1} e^{\mathrm i(\ell+1)\vartheta}+b_{\ell+1}e^{-\mathrm i(\ell+1)\vartheta})}{2^{n+2p}p!(\ell+1+p)!}r^{2p+\ell+1}\notag\\
&+\sum_{n=\ell+2}\sum_{p=0}^\infty\frac{(-1)^p(\sqrt{\gamma_1})^{2p+n}(a_{n} e^{\mathrm i n\vartheta}+b_{n}e^{-\mathrm in\vartheta}) }{2^{n+2p}p!(n+p)!}r^{2p+n}.\notag
\end{align}
From this, we can then derive the following integral equation
\begin{align}
&\frac{\eta (\sqrt{\gamma_1})^{\ell+1}}{2^{\ell +1}(\ell+1)!}\int_{\Gamma_{r_0}^\pm}(a_{\ell+1}e^{\mathrm i(\ell+1)\vartheta}+b_{\ell+1}e^{-\mathrm i(\ell+1)\vartheta})r^{\ell+1}e^{\rho\cdot \mathbf x}\mathrm d\sigma\notag\\
&\ =-\eta\int_{\Gamma_{r_0}^\pm}v\psi(\mathbf x)e^{\rho\cdot \mathbf x}\mathrm d\sigma+\int_{\mathcal S_{r_0}}(\gamma_1-\gamma_2)vu_0\mathrm d\mathbf x
+\int_{\Lambda_{r_0}}\partial vu_0-\partial_\nu uu_0\mathrm d\sigma\notag\\
&\ -\sum_{p=1}^\infty \frac{\eta(-1)^p (\sqrt{\gamma_1})^{2p+\ell+1}}{2^{\ell+1+2p}p!(\ell+1+p)!}\int_{\Gamma_{r_0}^\pm}(a_{\ell+1} e^{\mathrm i(\ell+1)\vartheta}+b_{\ell+1}e^{-\mathrm i(\ell+1)\vartheta})r^{2p+\ell+1}e^{\rho\cdot \mathbf x}\mathrm d\sigma\notag\\
&\ -\sum_{n=\ell+2}^\infty\sum_{p=0}^\infty\frac{\eta(-1)^p(\sqrt{\gamma_1})^{2p+n} }{2^{n+2p}p!(n+p)!} \int_{\Gamma_{r_0}^\pm} (a_{n} e^{\mathrm i n\vartheta}+b_{n}e^{-\mathrm in\vartheta})r^{2p+n}e^{\rho\cdot \mathbf x}\mathrm d\sigma.\notag
\end{align}
Then, by an asymptotic analysis similar to that used to prove $a_1=b_1=0$, we can conclude that
\begin{equation}\notag
\begin{cases}
(a_{\ell+1}e^{\mathrm i(\ell+1)\vartheta_m}+b_{\ell+1}e^{-\mathrm i(\ell+1)\vartheta_m})e^{-\mathrm i(\ell+2)(\vartheta_m-\varphi)}\\
\quad +(a_{\ell+1}e^{\mathrm i(\ell+1)\vartheta_M}+b_{\ell+1}e^{-\mathrm i(\ell+1)\vartheta_M})e^{-\mathrm i(\ell+2)(\vartheta_M-\varphi)}=0,\\
(a_{\ell+1}e^{\mathrm i(\ell+1)\vartheta_m}+b_{\ell+1}e^{-\mathrm i(\ell+1)\vartheta_m})e^{-\mathrm i(\ell+2)(\varphi-\vartheta_m)}\\
\quad +(a_{\ell+1}e^{\mathrm i(\ell+1)\vartheta_M}+b_{\ell+1}e^{-\mathrm i(\ell+1)\vartheta_M})e^{-\mathrm i(\ell+2)(\varphi-\vartheta_M)}=0,
\end{cases}
\end{equation}
which further induces that
\begin{equation}\notag
B\left(\begin{matrix}
a_{\ell+1}\\
b_{\ell+1}
\end{matrix}
\right)=\mathbf 0,
\end{equation}
where the coefficient matrix $B$ is
\begin{equation}\notag
B=\left(\begin{matrix}
e^{-\mathrm i\vartheta_m}+e^{-\mathrm i\vartheta_M}& e^{-(2\ell+3)\mathrm i\vartheta_m}+e^{-(2\ell+3)\mathrm i\vartheta_M}\\
e^{(2\ell+3)\mathrm i\vartheta_m}+e^{(2\ell+3)\mathrm i\vartheta_M}& e^{\mathrm i\vartheta_m}+e^{\mathrm i\vartheta_M}
\end{matrix}
\right).
\end{equation}
Using  $\mathcal S_{r_0}$ is  irrational, we know that $\vartheta_M-\vartheta_m\not=\frac{\alpha}{\ell+1}\pi,$ where $\alpha=1,2,\dots,\ell$, and $\vartheta_M-\vartheta_m\not=\frac{\sigma}{\ell+2}\pi,$ where $\sigma=1,2,\dots,\ell+1.$
Then the determinant of $B$ satisfies
$$\det(B)=2(\cos(\vartheta_M-\vartheta_m)-\cos((2\ell+3)(\vartheta_M-\vartheta_m)))\not=0,$$
which implies that $a_{\ell+1}=b_{\ell+1}=0$. {By mathematical induction,  we deduce that $v\equiv 0$ in $B_{r_0}$. By unique continuation principle for elliptic PDE, we know that $v\equiv 0$ in $D$. Due to the boundary conditions in \eqref{1-eq:scatter}, one has $w=\partial_\nu w=0$ on $\Gamma_{r_0}$, which implies that $w\equiv 0$ in $D$ by using unique continuation principle for elliptic equation.}

The proof is complete.
\end{proof}

In the following we will prove  Corollary \ref{thm:invis}.

\begin{proof}[The proof of Corollary \ref{thm:invis}]
We prove this theorem using a contradiction argument. Assume that $\Omega$ has an irrational convex corner and is invisible.  Since the Laplacian $\Delta$ is invariant under rigid motion, this corner of $\Omega$ can be described by $\mathcal S_{r_0}$ as given by \eqref{1eq:sr0}. Since $u^\infty\equiv 0$, by Rellich's theorem, one has the PDE system \eqref{eq:PB}, where $v=u^i$, $w=u^- \chi_{\Omega}+u^i\chi_{D\backslash \overline\Omega }$, $\gamma_1=k^2$ and  $\gamma_2=k^2q \chi_{\Omega}+k^2 \chi_{D\backslash \overline \Omega }$. Here $D:=B_R$ is chosen such that $\Omega\Subset D$. By applying Theorem \ref{thm:ucp}, we conclude that  $u^{i} \equiv 0$ in $D$,  which leads to a  contradiction by  the unique continuation principle.
\end{proof}

To address the inverse problem \eqref{eq:ip}, it is essential to define the admissible class of scatterers.
   Recall that Definition \ref{1-def:polygon} introduces the concept of an irrational polygon, which will be utilized in Definition \ref{1-def:com-ad}.

\begin{defn}\label{1-def:com-ad}
We define $\Omega$ as an \textit{admissible complex  scatterer} if it comprises finitely many disjointed scatterers, denoted as follows:
$$
(\Omega;q,\eta) =\bigcup _{\ell=1}^N(\Omega_\ell;q_\ell,\eta_{\ell}),\ \overline{\Omega_{\ell}}\cap\overline{\Omega_m}=\emptyset \ \mbox{for}\ \ell\not=m,$$
where $N\in \mathbb{N}$, and each $\Omega_{\ell}$ is a bounded Lipschtiz domain. Here, $q = \sum_{\ell=1}^N q_{\ell} \chi_{\Omega_{\ell}}$ with ${\rm supp}(q_\ell)\subset \Omega_\ell$ and $q_\ell \in H^2(\Omega_\ell )$. Additionally, $\eta = \sum_{\ell=1}^N \eta_{\ell} \chi_{\partial \Omega_{\ell}}$  and $\eta_\ell$ is a constant for $\ell=1, \ldots, N$.

Moreover, if each $\Omega_{\ell}$ is an irrational polygon, then $\Omega$ is considered to be an \textit{admissible complex  irrational polygonal scatterer.} Similarly, if each $\Omega_{\ell}$ is a convex irrational polygon, then $\Omega$ is considered to be an \textit{admissible complex  convex irrational polygonal scatterer.}
\end{defn}

The following theorem provides a local uniqueness result for determining the shape of a scatterer through a single far-field measurement, offering a more rigorous version of   Corollary \ref{thm:uni pre}. It asserts that the difference set between two admissible complex polygonal scatterers cannot contain a convex irrational corner if the far-field measurements associated with these two scatterers, for a fixed incident wave, are identical. Furthermore, we establish the global unique identification of the shape of admissible complex convex irrational polygonal scatterers using only a single far-field measurement.
\begin{thm}\label{1-thm:unique}

  Consider the conductive medium scattering described by \eqref{eq:sca} with the given incident wave $u^i$. Let $u_\ell^s \in H^1_{\rm loc}(\mathbb{R}^2)$ represent the scattered wave associated with $u^i$ and an admissible complex  polygonal scatterer $(\Omega_\ell;q_\ell,\eta_\ell)$ be defined in Definition \ref{1-def:com-ad}, where
\begin{equation}\notag
(\Omega_\ell;q_\ell,\eta_\ell)=\bigcup_{i=1}^{N_\ell}(\Omega_{i,\ell};q_{i,\ell},\eta_{i,\ell}),\quad N_{\ell}\in  \mathbb{N}_{+},\quad \ell=1,2.
\end{equation}
Let $u^\infty_{\ell}$ $(\ell=1,2)$, defined in \eqref{eq:asy}, be the far-field pattern corresponding to $u_\ell^s$, respectively.

If
\begin{equation}\label{eq:thm ass far}
u^{\infty}_1(\mathbf{\hat{x}};u^i)=u^{\infty}_2(\mathbf{\hat{x}};u^i),
\end{equation}
for all $\mathbf{\hat{x}}\in \mathbb{S}^1$ and a fixed incident wave $u^i$, then
{the difference $\Omega_1\aleph\Omega_2$ defined in Corollary \ref{thm:uni pre}} cannot contain a convex irrational corner.

Moreover,  if \eqref{eq:thm ass far} is satisfied and $\Omega_\ell$ is an admissible complex convex irrational polygonal scatterer $(\ell=1,2)$, then $\Omega_1=\Omega_2$. In other words, $N_1=N_2$ and $\Omega_{i,1}=\Omega_{i,2}$ for $i=1,\ldots, N_1$.
\end{thm}


\begin{proof}
We  prove the first conclusion of this theorem by contradiction. When \eqref{eq:thm ass far} holds, assume that {$\Omega_1\aleph\Omega_2$ } contains an irrational corner.  Underlying the fact that the Laplacian $\Delta$ is invariant under the rigid motion and by the definition of an admissible complex  irrational polygonal scatterer, we also assume that there exists an irrational corner $\mathcal S_{r_0}$ such that $\mathcal S_{r_0}\Subset \Omega_1\setminus \overline{\Omega_2}$, where  $r_0\in \mathbb R_{+}$ is sufficiently  small. Let $B_{r_0}$ be a disk with radius $r_0 \in \mathbb R_+$. 


Let $u_1(\mathbf x)$ and $u_2(\mathbf x)$ be total wave fields  associated $(\Omega_1;q_1,\eta_1)$ and $(\Omega_2;q_2,\eta_2)$, respectively. Then using \eqref{eq:thm ass far},  by Rellich lemma, we have $u_1=u_2$ in $\mathbb R^2\setminus\overline{\Omega_1\cup\Omega_2}$.  By virtue of the trace theorem, we have
\begin{align}\label{eq:u1u2 boundary}
	u^{-}_1=u^{+}_1=u^{+}_2\ \mbox{on}\ \Gamma_{r_0}^{\pm}.
\end{align}
Utilizing the conductive boundary conditions of \eqref{eq:sca} and \eqref{eq:u1u2 boundary}, since $u_1$ and $u_2$ satisfy \eqref{eq:sca},  we can therefore obtain that
\begin{equation}\notag
\begin{cases}
\Delta u^{-}_1+k^2(q_1 \chi_{\mathcal S_{r_0}}+ \chi_{B_{r_0}\backslash \overline{\mathcal S_{r_0}} })u^{-}_1=0,\hspace*{1.6cm} \mbox{in}\  B_{r_0},\\
\Delta u^{+}_2+k^2u^{+}_2=0,\hspace*{4.9cm}  \mbox{in}\  B_{r_0},\\
u^{-}_1=u^{+}_2,\ \partial_\nu u^{-}_1=\partial_\nu u^{+}_2+\eta u^{+}_2,\hspace*{2.6cm}  \mbox{on}\ \Gamma_{r_0}^\pm.
\end{cases}
\end{equation}
It is clear that $u_2$ is analytic in $\mathcal S_{r_0}$. By applying Theorem \ref{thm:ucp}, we have
$$
u_2\equiv 0\ \mbox{in}\ \mathcal S_{r_0}.
$$
Using the unique continuation principle and noting $u_2=u^i+u_2^s$, due to the fact that the scattered wave $u_2^s$ fulfills the Sommerfeld radiation condition, we get the contradiction.

The second conclusion can be directly obtained by using the definition of an admissible complex convex irrational polygonal scatterer.
\end{proof}



\section{Unique recovery results for the  conductive polygonal-nest and polygonal-cell medium scatterers}\label{sec:3}



In this section, we introduce the concepts of conductive polygonal-nest and polygonal-cell medium scatterers. We then formulate the corresponding inverse problems according to \eqref{eq:ip} when the scatterer in the direct scattering problem \eqref{eq:sca} exhibits a polygonal-nest or polygonal-cell structure. To establish the uniqueness result for determining the shape of the scatterer and its polygonal-nest or polygonal-cell structure using a single far-field measurement, we introduce the concept of the admissible class and prove that the gradient of the total wave field has H\"older regularity near the corners. Additionally, we demonstrate that the reflective index, represented by a linear polynomial, and the constant conductive boundary parameter can be uniquely identified through a single far-field measurement once the shape is determined.


\begin{defn}\label{1-def:nest}
Let $\Omega$ be a bounded, simply connected Lipschitz domain with a connected complement $\mathbb{R}^2 \setminus \overline{\Omega}$. We say that $\Omega$ has a polygonal-nest structure if there exist open, convex, simply connected polygons $\Sigma_i$, for $i = 1, \dots, N \in \mathbb{N}$, such that
$$
\Sigma_N \Subset \Sigma_{N-1}\Subset\dots\Subset\Sigma_2\Subset\Sigma_1=\Omega.
$$
\end{defn}

 A schematic illustration of a simply connected Lipschitz domain with a polygonal-nest structure is displayed in Figure \ref{1-fig:nest}.

\begin{defn} \label{1-def:cell}
Let $\Omega$ be a bounded, simply connected Lipschitz domain with a connected complement $\mathbb{R}^2\setminus\overline{\Omega}$. We say that $\Omega$ has a polygonal-cell structure if there exist open, convex, simply connected polygons $\Sigma_i$, for $i = 1, \dots, N \in \mathbb{N}$, satisfying the following conditions:
\begin{itemize}
\item[(a)] $\Sigma_i \subset \Omega$ and $\Sigma_i \cap \Sigma_j = \emptyset$ for $i \neq j$,
\item[(b)] $\cup_{i=1}^N \overline{\Sigma_i} = \overline{\Omega}$,
\item[(c)] For each $\Sigma_i$, there exists at least one vertex $\mathbf{x}_{p,i}$ such that the two edges of $\partial \Sigma_i$ associated with $\mathbf{x}_{p,i}$ are denoted by $\Gamma_i^{\pm} \subset \partial\Omega$.
\end{itemize}
\end{defn}
 A schematic illustration of a simply connected Lipschitz domain with a polygonal-cell structure is displayed in Figure \ref{1-fig:cell}.

\begin{figure}[ht!]
\centering
\begin{subfigure}[b]{0.34\textwidth}
\centering
\includegraphics[width=\linewidth]{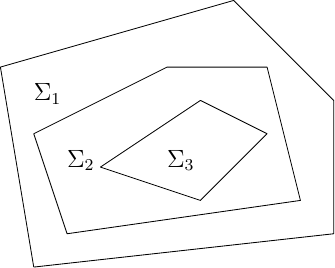}
\caption{A polygonal-nest structure.}
\label{1-fig:nest}
\end{subfigure}
\quad\quad\quad
\begin{subfigure}[b]{0.45\textwidth}
\centering
\includegraphics[width=\textwidth]{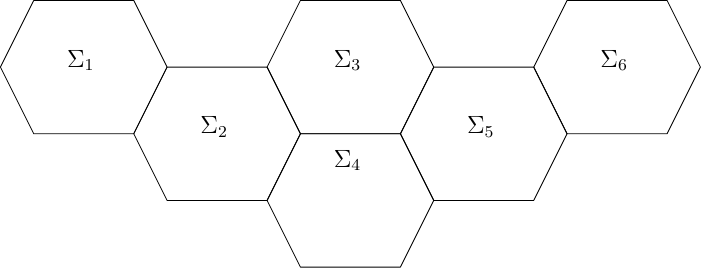}
\caption{A polygonal-cell structure.}
\label{1-fig:cell}
\end{subfigure}
\caption{The schematic illustrations of the polygonal-nest and the polygonal-cell  structures.  }
\end{figure}


\begin{defn}\label{1-def:pi-nest}
Let $(\Omega;q,\eta)$ be a conductive medium scatterer, if it satisfies the following conditions:
\begin{itemize}
\item[(a)] $\Omega$ has a polygonal-nest structure defined in Definition \ref{1-def:nest};
\item[(b)] each $\Sigma_i$ is a medium scatterer such that $\mathcal D_i:=\Sigma_i\setminus\overline{\Sigma_{i+1}}$ with physical configurations $q_i$ and $\eta_i$ (on $\partial \Sigma_i$), which is denoted as $(\mathcal D_i;q_i,\eta_i)$   with $\Sigma_{N+1}=\emptyset$, where $q_i$ is a  polynomial of order $1$ with the form given by
\begin{equation}\label{2eq:qform}
q_i=q_i^{(0)}+q_i^{(1)}x_1+q_i^{(2)}x_2,\quad q_i^{(j)}\in \mathbb C, \quad (j=0,1,2)
\end{equation}
in $\mathcal D_i$, and $\eta_i$ is a constant on $\partial \Sigma_i$ for $i=1,\dots,N$,
\end{itemize}
then we say that $(\Omega;q,\eta)$ possesses a polygonal-nest structure.
\end{defn}

Throughout of the rest paper, we rewrite a conductive polygonal-nest medium $(\Omega;q,\eta)$ given by Definition \ref{1-def:pi-nest} as
\begin{equation}\label{2eq:nest}
(\Omega;q,\eta)=\bigcup_{i=1}^N(\mathcal D_i;q_i,\eta_i)
\end{equation}
and $\Omega= \bigcup_{i=1}^N \mathcal D_i,\ q=\sum_{i=1}^Nq_{i}\chi_{\mathcal D_i},\eta=\sum_{i=1}^N\eta_i\chi_{\partial\Sigma_i}$.

Considering  the direct scattering problem \eqref{eq:sca} associated with  a conductive polygonal-nest medium scatterer $(\Omega;q,\eta)$ described by \eqref{2eq:nest}, let $u\in H^1_{loc}(\mathbb R^2)$ be the corresponding total wave field. In view of the polygonal-nest structure of $(\Omega;q,\eta)$, the  conductive transmission boundary conditions of the total wave $u$ across $\partial \Sigma_{i+1}$ are given by
\begin{equation}\label{2eq:conb}
u_{i}|_{\partial \Sigma_{i+1}}=u_{i+1}|_{\partial \Sigma_{i+1}},\ (\partial_{\nu} u_{i}+\eta_{i+1}u_{i})|_{\partial \Sigma_{i+1}}=\partial _\nu u_{i+1}|_{\partial \Sigma_{i+1}},
\end{equation}
where $u_i=u|_{\mathcal D_{i}},i=0,1,\dots,N-1$ and $\Sigma_{0}=\mathbb R^2\setminus \overline{\Sigma_1}$. Indeed, the well-posedness of the direct scattering  problem can be obtained by using a similar variational argument in \cite{Bon-Liu}, where the detailed proof is omitted. In this paper, our main focus is the corresponding inverse problems associated with  a conductive polygonal-nest  medium scatterer $(\Omega;q,\eta)$, which aims to determine the polygonal-nest structure  of $\Omega$ and the physical configurations  $q_\ell$ and $\eta_\ell$ from the knowledge of the far-field pattern $u^{\infty}(\mathbf {\hat{x}};u^i)$ with a fixed incident wave $u^i$. Here, $u^{\infty}(\mathbf {\hat{x}};u^i)$ is the far-field pattern of $u$ given by \eqref{eq:asy}.  The above inverse problems can be formulated by
 \begin{align}\label{eq:IP 1}
 	 u^{\infty}(\mathbf {\hat{x}};u^i) \rightarrow \bigcup_{\ell=1}^N(\mathcal D_\ell;q_\ell,\eta_\ell),\ \forall \mathbf {\hat x}\in \mathbb S^{1}.
 \end{align}

Similarly, we can define a conductive polygonal-cell medium scatterer $\Omega$ as follows.
\begin{defn}\label{1-def:pi-cell}
Let $(\Omega;q,\eta)$ be a conductive medium scatterer, if the following conditions are fulfilled:
\begin{itemize}
\item[(a)] $\Omega$ has a polygonal-cell structure defined in Definition \ref{1-def:cell};
\item[(b)] each $\Sigma_i$ is a medium scatterer with physical configurations $q_i$ and $\tilde \eta$ and is denoted as $(\Sigma_i;q_i,\tilde\eta)$, $q_i$ is a  polynomial of order $1$ in the form of \eqref{2eq:qform} in $\Sigma_i$ and $\tilde \eta$ is a constant on $\partial \Sigma_i$ for $i=1,\dots,N$,
\end{itemize}
then $(\Omega;q,\eta)$ is said to have a polygonal-cell structure.
\end{defn}
 We  write a conductive  polygonal-cell  medium scatterer $(\Omega;q,\eta)$ as
\begin{align}\label{eq:cell partion}
	(\Omega;q,\eta)=\bigcup_{i=1}^N(\Sigma_i;q_i,\tilde \eta),
\end{align}
where $\Omega=\cup_{i=1}^N\Sigma_{i},\ q=\sum_{i=1}^Nq_i\chi_{\Sigma_i},\ \eta=\sum_{i=1}^N\tilde \eta \chi_{\partial\Sigma_i}.$ The  inverse problems for a conductive polygonal-cell medium scatterer by a single far-field measurement can be described in a similar way as for the case with a conductive polygonal-nest  medium scatterer. Namely, one aims to  determine the shape of $\Omega$ and physical configurations $q_i$ and $\tilde \eta$ through the measurement	 $u^{\infty}(\mathbf {\hat{x}};u^i)$ associated with a fixed incident wave $u^i$, which can be formulated as follows:
\begin{align}\label{eq:IP 2}
 	 u^{\infty}(\mathbf {\hat{x}};u^i) \rightarrow \bigcup_{i=1}^N(\Sigma_i;q_i,\tilde \eta),\ \forall \mathbf {\hat x}\in \mathbb S^{1}.
\end{align}
Here, $u^{\infty}(\mathbf{\hat{x}};u^i)$ is the far-field pattern given by \eqref{eq:asy}, associated with the total wave $u$, where $u$ is the solution of \eqref{eq:sca} corresponding to the conductive polygonal-cell medium scatterer $\bigcup_{i=1}^N(\Sigma_i;q_i,\tilde{\eta})$.

In order to introduce the admissible class for studying the inverse problems of the unique identification of the shape, we need to study the H\"older-regularity of the solution of the coupled PDE system \eqref{lem:31 eq} near the corner, where $u_1$ and $u_2$ are the total wave field corresponding to the two different scatterers $\Omega_1$ and $\Omega_2$, respectively. Indeed, the PDE system \eqref{lem:31 eq} is derived from the proceed of proving the unique determination of the shape of the scatterer by contradiction (see \eqref{eq:312 contract} in the proof of Theorem \ref{2thm:un-shape}). Recall that $\mathcal{S}_{r_0}$ is defined by \eqref{1eq:sr0} with two line segment boundaries $\Gamma^{\pm}_{r_0}$.
We can similarly define $\mathcal{S}_{t r_0}$ for $t \in \mathbb{R}_{+}$.

\begin{lem}\label{2lem:reggularity}
Let the corner $\mathcal S_{2r_0}=\mathcal S \cap B_{2r_0} $, where $\mathcal S$ is given by \eqref{eq:s cor} and $B_{2r_0} $ is a disk centered at the origin with the radius $2r_0$ ($r_0\in \mathbb R_+$). Denote $\Gamma^\pm_{2r_0}=\Gamma^{\pm}\cap B_{2r_0}$, where $\Gamma^{\pm}$ is the boundary line of $\mathcal S$. Suppose that $u_1,u_2\in H^1(\mathcal S_{2r_0})$ satisfy
\begin{equation}\label{lem:31 eq}
\begin{cases}
(\Delta+k^2q)u_1=0\hspace*{1.0cm} \mbox{in}\ \mathcal S_{2r_0},\\
(\Delta+k^2)u_2=0\hspace*{1.20cm} \mbox{in}\ B_{2r_0},\\
u_1=u_2\hspace*{2.4cm} \mbox{on}\ \Gamma_{2r_0}^\pm, 
\end{cases}
\end{equation}
where $k$ is a positive constant and $q$ is in the form of \eqref{2eq:qform}. 
There exists a positive number  $t_0'\in \mathbb R_+$
such that  $u_1\in C^{1,\alpha}(\overline{\mathcal S_{t_0'r_0}})$, where  $\alpha\in(0,1)$ depends on the opening angle of $\mathcal S_{2r_0}$.
\end{lem}

\begin{proof}
{
For any given $t \in \mathbb{R}_{+}$, let $\mathbf{y}_{t,+} \in \Gamma^{+}$ and
$\mathbf{y}_{t,-} \in \Gamma^{-}$ be points satisfying $|\mathbf{y}_{t,+}| =
|\mathbf{y}_{t,-}| = t$, where $\Gamma^{+}$ and $\Gamma^{-}$ are defined by
\eqref{eq:gamma+} and \eqref{eq:gamma-}, respectively. For any $t \in \mathbb{R}_{+}$, we introduce the triangle $V_{t r_0}$ formed by the three line segments connecting  $\mathbf{0}$, $\mathbf{y}_{t,+}$, and $\mathbf{y}_{t,-}$.

Let $\varphi\in C^\infty({\mathcal S_{2r_0}})$ be a smooth cutoff function satisfying that $\varphi\equiv1$ in ${\mathcal S_{t_0r_0}}$ and $\varphi\equiv 0$ in $\mathcal S_{2 r_0}  \setminus\overline{ \mathcal S_{t_0r_0}}$, where $t_0\in (0,2)$ such that $\mathcal S_{t_0 r_0} \Subset V_{2r_0}$. Let $w=u_1-u_2$ and define $\tilde{\varphi}=w\varphi$. Then there holds that
$$
\Delta \tilde{\varphi}=f\ \mbox{in}\   V_{2r_0},\quad
\tilde{\varphi}=0\ \mbox{on}\  \partial {V_{2{r_0}}},
$$
where
$f=\Delta w\varphi+w\Delta\varphi +2\nabla w\cdot \nabla \varphi
=(k^2u_2-k^2qu_1)\varphi+ w\Delta\varphi +2\nabla w\cdot \nabla \varphi\notag.$    Since $u_1\in H^1(\mathcal S_{2r_0})$ and $u_2$ is analytic in $\mathcal S_{2r_0}$, it yields that  $w\in H^1(\mathcal S_{2r_0})$, $\nabla w\in L^2(\mathcal S_{2r_0})$ and $f\in L^2(\mathcal S_{2r_0})$. Moreover, it is noted that the opening angle of the sector $\vartheta_M-\vartheta_m\in (0,\pi)$ and $\tilde {\varphi}\in H^1_0( V_{2r_0})$. Applying  \cite[Theorem 10.28]{DN} to the case of the polygonal network consists of only one subset $V_{2r_0}$,  then we have $\tilde{\varphi}\in H^2(V_{2r_0})$. Consequently, we have  $w\in H^2(\mathcal S_{t_0 r_0})$ since $\varphi\equiv1$ in $\mathcal S_{t_0r_0}$. Therefore, we conclude that $u_1\in H^2(\mathcal S_{t_0 r_0})$.

Now we focus on the triangular domain \( V_{t_0 r_0} \). Let \( \psi \in C^\infty(\mathcal{S}_{t_0 r_0}) \) be a smooth cutoff function satisfying:
\[
\psi \equiv 1 \text{ in } \mathcal{S}_{t_0' r_0} \quad \text{and} \quad \psi \equiv 0 \text{ in } \mathcal{S}_{t_0 r_0} \setminus \mathcal{S}_{t_0' r_0},
\]
where \( t_0' \) is fixed such that \( 0 < t_0' < t_0 \) and \( \mathcal{S}_{t_0' r_0} \Subset V_{t_0 r_0} \).  Similarly, denote $\tilde{\psi}=w\psi$,  we then have
\begin{align}
\Delta \tilde{\psi}=g\ \mbox{in}\ V_{t_0 r_0},\quad \tilde{\psi}=0\ \mbox{on}\ \partial V_{t_0 r_0},\label{eq:tildepsi}
\end{align}
where
$
g=(k^2u_2-k^2qu_1)\psi+ w\Delta\psi +2\nabla w\cdot \nabla \psi.
$
Recall that   $u_1\in H^2 (\mathcal S_{t_0 r_0}) $,  one has $w\in H^2(V_{t_0 r_0} )$ and $\nabla w\in H^1(V_{t_0 r_0})$, which indicates that  $g\in H^1(V_{t_0 r_0})$. According to \cite[Theorem 2.8]{AD},  we have $g\in W^{1,p}(V_{t_0 r_0}),\ 1<p<2$.
Moreover,   from \cite[Theorem 2.1]{MZ92}, we know that  $\tilde {\psi}$ satisfying \eqref{eq:tildepsi} admits the   decomposition  as follows:
\begin{equation}\notag
\tilde {\psi}({\bf x})=\tilde{\psi}_{\rm reg}({\bf x})+\sum_{0<\lambda_{\ell }<3-\frac{2}{p}} \zeta_\ell r^{\lambda_{\ell }}\sin(\lambda_{\ell }\vartheta),\ \tilde {\psi}_{\rm reg} \in W^{3,p}(V_{2r_0}),\ {\bf x}=r(\cos \vartheta, \sin  \vartheta),
\end{equation}
where $\zeta$ is a constant depending only on $f$ and $\varphi$, and $\lambda_{\ell }=\frac{\ell \pi}{\vartheta_M-\vartheta_m},\ \ell\in \mathbb N_{+}$.
If $\vartheta_M-\vartheta_m\in (0,\pi/2]$, then we know that $\tilde {\psi}({\bf x})=\tilde {\psi}_{\rm reg}({\bf x})$ since $\lambda_\ell\geq 2\ell>3-\frac{2}{p}$. If $\vartheta_M-\vartheta_m\in (\pi/2,\pi )$, we then choose  $p\in(1,\frac{2}{3-\frac{\pi}{\vartheta_M-\vartheta_m}})$ since it can be verified that $\frac{2}{3-\frac{\pi}{\vartheta_M-\vartheta_m}}\in (1,2)$, which also  indicates that $\lambda_\ell>3-\frac{2}{p}$ for any $\ell \in \mathbb N_+$. Therefore,  $\tilde {\psi}=\tilde {\psi}_{\rm reg}\in W^{3,p}(V_{t_0 r_0})$ for $p\in (1,2)$. Recall that $\psi\equiv 1$ in $\mathcal S_{t_0' r_0}$, then we deduce that $w\in W^{3,p}(\mathcal S_{t_0' r_0})$ which implies  that $u_1\in W^{3,p}(\mathcal S_{t_0' r_0})$ since $u_2$ is analytic in $\mathcal S_{t_0' r_0}$. It follows from  the Sobolev embedding theorem (cf. \cite[Theorem 7.26]{GT-ell}), we have that $u_1\in C^{1,\alpha}(\overline {\mathcal S_{t_0' r_0}})$ with  $\alpha\in(0,1)$, where $\alpha$ depends on the opening angle of the sector $\mathcal S_{2r_0}$.

}

The proof is complete.
\end{proof}


In the following, we will introduce the admissible class for studying the inverse problems \eqref{eq:IP 1} and \eqref{eq:IP 2}, which will be used to establish recovery uniqueness results in the subsequent theorems.


\begin{defn}\label{2def:adm}

Suppose $(\Omega;q,\eta)$ represents a conductive polygonal-nest or polygonal-cell medium scatterer as defined in  Definitions \ref{1-def:pi-nest} or \ref{1-def:pi-cell}. Let $u$ denote the total wave field of  \eqref{eq:sca} associated with  $(\Omega;q,\eta)$.

\begin{itemize}

\item Polygonal-Nest Admissibility:

\begin{enumerate}
	
\item For a  conductive polygonal-nest medium scatterer $(\Omega;q,\eta)$, it is considered admissible if it satisfies one of the following conditions:
\begin{itemize}
\item[(I)]
		{\begin{equation}\label{2eq:ad-1}
      \lim_{\rho\to +0}\frac{1}{m( B(\mathbf x_p,\rho))}\int_{B(\mathbf x_p,\rho)} \vert u(\mathbf x)\vert \mathrm d\mathbf x\not =0,
\end{equation}}
\item[(II)]{if
 \begin{align}
 \lim_{\rho\to +0}\frac{1}{m( B(\mathbf x_p,\rho))}\int_{B(\mathbf x_p,\rho)} \vert u(\mathbf x)\vert\mathrm d\mathbf x =0,\  \mbox{then}\notag\\
 \lim_{\rho\to +0}\frac{1}{m( B(\mathbf x_p,\rho))}\int_{B(\mathbf x_p,\rho)} \vert \nabla u(\mathbf x)\vert \mathrm d\mathbf x\neq0,   \label{2eq:ad-2}
 \end{align}}
 \end{itemize}
{where $m(B(\mathbf x_p,\rho))$ is the measure of $B(\mathbf x_p,\rho)$,}
 $\mathbf{x}_{p} \in {\mathcal{V}}( \Sigma_i)$ and ${\mathcal V}( \Sigma_i)=\{\mathbf x_{p,i}\}_{p=1}^{\ell_i} (3\leq\ell_{i}\in \mathbb Z_{+})$ is the vertex set of the convex polygon $\Sigma_i$. Here, $\{\Sigma_i\}_{i=1}^N$ is the polygonal-nest structure of $\Omega$ described by Definition \ref{1-def:nest}.

\item If the condition \eqref{2eq:ad-2} is satisfied, it is additionally   {required that $\angle( \Gamma_{p}^{-},\Gamma_{p}^{+})\in (0,\pi)\setminus \{\frac{\pi}{2}\}$, }where $\Gamma_{p}^{\pm}$ represent two edges of $\Sigma_i$, and $\Gamma_{p}^{-}\cap\Gamma_{p}^{+}=\mathbf{x}_{p}$ with $\angle( \Gamma_{p}^{-},\Gamma_{p}^{+})$ being the intersection angle between $\Gamma_{p}^{-}$ and $\Gamma_{p}^{+}$.
\end{enumerate}

\item Polygonal-Cell Admissibility:
\begin{enumerate}

\item
For a conductive polygonal-cell medium scatterer $(\Omega;q,\eta)$, it is considered admissible if it satisfies condition \eqref{2eq:ad-1} or \eqref{2eq:ad-2}, where $\mathbf{x}_{p}\in {\mathcal{V}}( \Omega ) $. Here, ${\mathcal{V}}( \Omega ) $ represents the vertex set of the polygon $\Omega$.

\item Similarly, if the condition \eqref{2eq:ad-2} is satisfied, it is additionally   required that {$\angle( \Gamma_{p}^{-},\Gamma_{p}^{+})\in (0,\pi)\setminus \{\frac{\pi}{2}\}$}, where $\Gamma_{p}^{\pm}$ represent two edges of $\Omega$, and $\Gamma_{p}^{-}\cap\Gamma_{p}^{+}=\mathbf{x}_{p}$ with $\angle( \Gamma_{p}^{-},\Gamma_{p}^{+})$ being the intersection angle between $\Gamma_{p}^{-}$ and $\Gamma_{p}^{+}$.
\end{enumerate}
\end{itemize}

\end{defn}

\begin{rem}\label{rem31}

The non-void assumptions \eqref{2eq:ad-1} and \eqref{2eq:ad-2} in Definition \ref{2def:adm} represent generic physical conditions. Indeed, the generic assumptions \eqref{2eq:ad-1} and \eqref{2eq:ad-2}  can be fulfilled in certain physical scenarios. For example, when
\begin{align}\label{eq:ad k diam}
k \cdot {\rm diam}(\Omega) \ll 1,
\end{align}
from a physical perspective, the scattered wave $u^s$ can be neglected. In fact, as rigorously justified in \cite{Bon-Liu}, it holds that
$$ \|u^s\|_{H^1(B_R)}\leq C(\|\eta u^i\|_{H^{-\frac{1}{2}}(\partial \Omega)}+k^2\|qu^i\|_{L^2(\Omega)}),\quad \Omega \subset B_R, 
$$
where $C$ is a positive number and $B_R$ is a ball centered at the origin with radius $R$. Hence, if \eqref{eq:ad k diam} is satisfied and the conductive boundary parameter $\eta$ is sufficiently small, the incident wave $u^i$ dominates the total wave $u=u^i+u^s$. It is evident that the  incident wave $u^i$ fulfills \eqref{2eq:ad-1} (or \eqref{2eq:ad-2}) , which implies that the total wave $u$ generically satisfies \eqref{2eq:ad-1} or \eqref{2eq:ad-2}. On the other hand, from a physical standpoint, in conductive medium scattering, if the total wave does not satisfies \eqref{2eq:ad-1} at a corner point, it implies that the scattered wave and the incident wave cancel each other at those points. Similarly, if the admissible condition \eqref{2eq:ad-2} is violated,   it indicates that either the scattered wave and the incident wave cancel each other, or the gradient of the scattered wave and the incident wave cancel each other at those points. Such cancellation phenomena, as described above, are rare in typical scattering scenarios. However, we will not delve deeply into this matter in this paper.

\end{rem}

\begin{rem}\label{rem32}
	In previous works such as \cite{DCL21, CDL20}, the admissible condition \eqref{2eq:ad-1} plays a crucial role in deriving corresponding uniqueness results. However, in this paper, when \eqref{2eq:ad-1} is violated, we propose another admissible condition \eqref{2eq:ad-2} {to derive the unique determination of the shape of the conductive medium scatterer.}  There are three distinct cases for \eqref{2eq:ad-2}, characterized as follows:
\begin{itemize}
{
\item[a)] $\lim_{\rho\to +0}\frac{1}{m( B(\mathbf x_p,\rho))}\int_{B(\mathbf x_p,\rho)} \vert \partial_1u(\mathbf{x})\vert \mathrm d\mathbf x\neq 0$, \\
    $\lim_{\rho\to +0}\frac{1}{m( B(\mathbf x_p,\rho))}\int_{B(\mathbf x_p,\rho)} \vert u(\mathbf{x})\vert \mathrm d\mathbf x=\lim_{\rho\to +0}\frac{1}{m( B(\mathbf x_p,\rho))}\int_{B(\mathbf x_p,\rho)} \vert \partial_2u(\mathbf{x}) \vert \mathrm d\mathbf x =0$,
\item[b)] $\lim_{\rho\to +0}\frac{1}{m( B(\mathbf x_p,\rho))}\int_{B(\mathbf x_p,\rho)}\vert \partial_2u(\mathbf{x}) \vert \mathrm d\mathbf x\neq 0$, \\
    $\lim_{\rho\to +0}\frac{1}{m( B(\mathbf x_p,\rho))}\int_{B(\mathbf x_p,\rho)}\vert u(\mathbf{x}) \vert \mathrm d\mathbf x=\lim_{\rho\to +0}\frac{1}{m( B(\mathbf x_p,\rho))}\int_{B(\mathbf x_p,\rho)}\vert \partial_1u(\mathbf{x})\vert \mathrm d\mathbf x=0$,
\item[c)] $\lim_{\rho\to +0}\frac{1}{m( B(\mathbf x_p,\rho))}\int_{B(\mathbf x_p,\rho)}\vert  \partial_1u(\mathbf{x})\vert \mathrm d\mathbf x\neq 0$,\\
     $\lim_{\rho\to +0}\frac{1}{m( B(\mathbf x_p,\rho))}\int_{B(\mathbf x_p,\rho)}\vert \partial_2u(\mathbf{x})\vert \mathrm d\mathbf x\neq 0$, $\lim_{\rho\to +0}\frac{1}{m( B(\mathbf x_p,\rho))}\int_{B(\mathbf x_p,\rho)}\vert u(\mathbf{x})\vert\mathrm d\mathbf x =0$.}
\end{itemize}
Since \eqref{2eq:ad-2} encompasses more cases compared to \eqref{2eq:ad-1}, it is more likely to be satisfied in practical scenarios. This increased applicability of \eqref{2eq:ad-2} enhances the practical relevance of our uniqueness results.
\end{rem}



In Theorem \ref{2thm:un-shape}, a local uniqueness result for determining the shape of an admissible conductive polygonal-nest or polygonal-cell medium scatterer by a single far-field measurement is established. Furthermore, when the admissible scatterer has a polygonal-nest structure, a global shape identifiability by a single far-field measurement is obtained.

\begin{thm}\label{2thm:un-shape}
Let $(\Omega_\ell;q_\ell,\eta_\ell),\ell=1,2$ be both admissible conductive polygonal-nest or both polygonal-cell medium scatterers satisfying \eqref{eq:sca}, $u^\infty_\ell(\mathbf {\hat{x}};u^i)$ be the far-field pattern associated with the scatterers $(\Omega_\ell;q_\ell,\eta_\ell)$ and the incident field $u^i$. If
$$
u^\infty_1(\mathbf {\hat{x}};u^i)=u^\infty_2(\mathbf {\hat{x}};u^i)
$$
is fulfilled for all $\mathbf {\hat {x}}\in \mathbb S^1$ and a fixed incident wave $u^i$, then {the difference $\Omega_1\aleph\Omega_2$ defined in Corollary \ref{thm:uni pre}} cannot have a convex corner. Moreover, if $\Omega_1$ and $\Omega_2$ are both admissible conductive polygonal-nest medium scatterers, one holds that
$$\partial \Omega_1=\partial \Omega_2.$$
\end{thm}




\begin{proof}
We only present the proof for the polygonal-nest case, since the polygonal-cell case can be proven in a similar manner. We proceed to prove this theorem by contradiction. Let us assume that there exists a corner in {$\Omega_1\aleph\Omega_2$}. According to \eqref{2eq:nest}, we can express the admissible conductive polygonal-nest medium scatterers as $(\Omega_\ell;q_\ell,\eta_\ell)$ ($\ell=1,2$) as follows
\begin{equation}\notag
(\Omega_\ell;q_\ell,\eta_\ell)=\bigcup_{i=1}^{N_\ell}(\mathcal D_{i,\ell};q_{i,\ell},\eta_{i,\ell}),
\end{equation}
and
\begin{equation}\label{2eq:nest2}
\Omega_\ell=\bigcup_{i=1}^{N_\ell}\mathcal D_{i,\ell},\ q_j=\sum_{i=1}^{N_\ell}q_{i,\ell}\chi_{\mathcal D_{i,\ell}},\ \eta_{\ell}=\sum_{i=1}^{N_\ell}\eta_{i,\ell}\chi_{\partial \Sigma_{i,\ell}}.
\end{equation}
Since $-\Delta$ is invariant under rigid motion,  without loss of generality, let us assume that $\bf 0$ is the vertex of the corner $\Omega_1 \cap \mathcal{S}$ such that $\mathbf 0 \in \partial \Omega_1$ but $\mathbf 0 \notin \partial \overline{\Omega_2}$, where $\mathbf{0}$ represents the origin. Moreover, we assume that the corner $\Omega_1 \cap \mathcal{S} = \mathcal{S}_{r_0} \Subset \Sigma_{1,1}$, where $\Sigma_{1,1}$ is defined in Definition \ref{1-def:nest} and \eqref{2eq:nest2}.

Let $u_1(\mathbf x)$ and $u_2(\mathbf x)$ be the total wave fields associated with $\Omega_1$ and $\Omega_2$, respectively. Given   $u^{\infty}_1=u_2^{\infty}$ for all $\mathbf {\hat{x}}\in \mathbb S^1$ and by Rellich lemma,  we conclude that  $u^s_1=u^s_2$ in $\mathbb R^2\setminus (\overline {\Omega_1\cup \Omega_2})$. This implies  that
$u_1(\mathbf x)=u_2(\mathbf x)$ for all $\mathbf {\hat{x}}\in  \mathbb R^2\setminus (\overline {\Omega_1\cup \Omega_2})$. Then  $u_1^{-}$ and $u_2^{+}$ satisfy the following PDE system,
\begin{equation}\label{eq:312 contract}
\begin{cases}
\Delta u_1^{-}+k^2q_1u_1^{-}=0,\hspace*{2.6cm}\mbox{in}\ \mathcal S_{r_0},\\
\Delta u_2^{+}+k^2u_2^{+}=0,\hspace*{2.95cm}\mbox{in}\ \mathcal S_{r_0},\\
u_1^{-}=u_2^{+},\ \partial_{\nu}u_1^{-}=\partial_\nu u_2^{+}+\eta_1u_2^{+},\hspace*{0.5cm}\mbox{on}\ \Gamma_{r_0}^\pm.
\end{cases}
\end{equation}
Let $\bar{u}=u_1^{-}-u_2^{+}$. Then it yields that
\begin{equation}\notag
\begin{cases}
\Delta \bar{u}+k^2q_1 \bar{u}=k^2(1-q_1)u_2^{+}, \hspace*{0.5cm}\mbox{in}\ \mathcal S_{r_0},\\
\bar{u}=0,\ \partial_\nu\bar{u}=\eta_1{ u_1^{-}}, \hspace*{1.95cm}\mbox{on}\ \Gamma_{r_0}^{\pm}.
\end{cases}
\end{equation}
Since $q_1\in H^2(\mathcal S_{r_0})$, let $\widetilde q_1$ be the Sobolev extension of $q_1$ in $\mathbb R^2$. According to Lemma \ref{1-lem:cgo}, there exists the CGO solution $u_0$ given by \eqref{1-eq:cgo}, which satisfies \eqref{eq:cgo par}. Hence it yields that
\begin{equation}\notag 
(\Delta+k^2 q_1)u_0=0\mbox{ in } \mathcal S_{r_0}.
\end{equation}

By virtue of the CGO solution defined in \eqref{1-eq:cgo}, using Green's formula, one has
\begin{align}
\eta_1\int_{\Gamma_{r_0}^\pm}u_1^{-}u_0\mathrm d\mathbf x=k^2\int_{\mathcal S_{r_0}}(1-q_1)u_1^{-}u_0\mathrm d\mathbf x+\int_{\Lambda_{r_0}}\partial_{\nu}u_0\bar{u}-\partial_\nu\bar{u}u_0\mathrm d\sigma.\label{2eq:baru1}
\end{align}
According to Lemma \ref{2lem:reggularity}, there exists a positive number $t_0'\in \mathbb R_+$ such that  $u_1^{-}\in C^{1,\alpha}(\overline{\mathcal S_{t_0'r_0}}),\alpha\in(0,1)$, which implies the following expansion
\begin{equation}\label{2eq:uext}
u_1^{-}=u_1^{-}(\mathbf 0)+\partial_1 u_1^{-}(\mathbf 0) x_1+\partial_2u_1^{-}(\mathbf 0)x_2 +\delta u_1^{-},\ \vert \delta u_1^{-}\vert\leq \vert \mathbf x\vert ^{1+\alpha }\|u_2^{+}\|_{C^{1,\alpha}}.
\end{equation}

In the subsequent analysis, we need to establish an integral identity in $\mathcal{S}_{t_0'r_0}$ that is analogous to \eqref{2eq:baru1}. However, our analysis indicates that the parameter $t_0'$ defining $\mathcal{S}_{t_0'r_0}$ does not play a crucial role in proving this theorem. Therefore, we will instead use \eqref{2eq:baru1} for our further analysis.

According to the admissibility conditions in Definition \ref{2def:adm}, we will divide the proof into four cases, as follows.

\medskip
\noindent\textbf{Case (a)}: $u_1^{-}(\mathbf 0)\not =0$. Combing \eqref{2eq:uext} with \eqref{2eq:baru1}, we can deduce that
\begin{align}
\eta_1u_1^{-}(\mathbf 0)&\int_{\Gamma_{r_0}^\pm}e^{\rho\cdot \mathbf x}\mathrm d\mathbf x=\sum_{i=1}^4I_i,\label{2eq:u-0}
\end{align}
where
\begin{align}
I_1&=-\eta_1u_1^{-}(\mathbf 0)\int_{\Gamma_{r_0}^\pm}\psi(\mathbf x)e^{\rho\cdot \mathbf x}\mathrm d\sigma,\ I_2=k^2\int_{\mathcal S_{r_0}}(1-q_1)u_1^{-}u_0\mathrm d\mathbf x
\notag\\
I_3&=-\eta_1\int_{\Gamma_{r_0}^\pm}(\partial_1 u_1^{-}(\mathbf 0) x_1+\partial_2u_1^{-}(\mathbf 0)x_2 + \delta u_1^{-})u_0\mathrm d\sigma,\notag\\ I_4&=\int_{\Lambda_{r_0}}\partial_{\nu}u_0\bar{u}-\partial_\nu\bar{u}u_0\mathrm d\sigma.\label{2eq:u-01}
\end{align}
By \eqref{1-eq:someest}, we have
\begin{align}
\vert I_1\vert &\lesssim \tau^{-\frac{7}{6}}+\mathcal O(\tau^{-\frac{7}{6}}e^{-\frac{1}{2}\varsigma r_0\tau}),\ \vert I_2\vert \lesssim \tau^{-2}+\tau^{-\frac{29}{12}}+\mathcal O(\tau^{-1}e^{-\frac{1}{2}\varsigma\tau}),\notag\\
\vert I_3\vert &\lesssim\tau^{-2}+\tau^{-(2+\alpha )}+\mathcal O(\tau^{-1}e^{-\frac{1}{2}\varsigma r_0\tau}),\ \vert I_4\vert \lesssim (1+\tau)(1+\tau^{-\frac{2}{3}})e^{-\varsigma r_0\tau}.\label{2eq:u-02}
\end{align}
Let $\mathbf d$ be given in the form of \eqref{1-eq:d2}. We can then directly derive that
\begin{align}
\vert \eta_1u_1^{-}(\mathbf 0)\vert&\left\vert \int_{\Gamma_{r_0}^\pm} e^{\rho\cdot\mathbf x}\mathrm d\sigma \right\vert \gtrsim
\frac{\Gamma(1)\vert \eta_1u_2^{+}(\mathbf 0)\vert\vert e^{\mathrm i\vartheta_m}+e^{\mathrm i\vartheta_M}\vert }{|\tau e^{\mathrm i\varphi}|}-\mathcal O(\tau^{-1}e^{-\frac{1}{2}\lambda r_0 \tau}),\label{2eq:u-03}
\end{align}
here and in what follows, the term $ \mathcal O(\tau^{-1}e^{-\frac{1}{2}\lambda r_0 \tau})$ is defined in \eqref{1-eq:1est}.
Recall that $\mathcal S_{r_0}$ is defined by \eqref{1eq:sr0}, where the polar angle of $\Gamma_{r_0}^-$ is $\vartheta_m$ and the polar angle of $\Gamma_{r_0}^+$ is $\vartheta_M$.
Combining \eqref{2eq:u-01}- \eqref{2eq:u-03} with \eqref{2eq:u-0}, then multiplying $\tau$  and letting $\tau\to \infty$, we see that
$$
\vert \eta_1u_1^{-}(\mathbf 0) \vert \vert  e^{\mathrm i\vartheta_m}+e^{\mathrm i\vartheta_M}\vert=0.
$$
Since
$$\vert e^{\mathrm i \vartheta_m}+ e^{\mathrm i \vartheta_M}\vert =\left(2(\cos(\vartheta_M-\vartheta_m)+1)\right)^{1/2}\not=0 \quad \mbox{for}\quad\vartheta_M-\vartheta_m\in(0,\pi),$$
it follows that
$$ \vert \eta_1u_1^{-}(\mathbf 0) \vert =0. $$
 However, this contradicts to the admissible condition $u_1^{-}(\mathbf 0)\neq0$  due to $\eta_1\neq0$.

\medskip
\noindent\textbf{Case (b)}:  $\partial_1u_1^{-}(\mathbf 0)\not=0$, $u_1^{-}(\mathbf 0)=\partial_2u_1^{-}(\mathbf 0)=0$. Similar to \eqref{2eq:u-0}, we have the following integral equality:
\begin{align}
\eta_1\partial_1u_1^{-}(\mathbf 0)&\int_{\Gamma_{r_0}^\pm} e^{\rho\cdot \mathbf x}x_1\mathrm d \sigma= -\eta_1\partial_1u_1^{-}(\mathbf 0)\int_{\Gamma_{r_0}^\pm}\psi(\mathbf x)e^{\rho\cdot \mathbf x}x_1\mathrm d\sigma \notag\\
&+\int_{\mathcal S_{r_0}}(1-q_1)u_1^{-}u_0\mathrm d\mathbf x-\eta_1\int_{\Gamma_{r_0}^\pm}\delta u_1^{-}u_0\mathrm d\sigma+I_4\label{2eq:u-04}.
\end{align}
By direct calculations we know that the right-hand side of \eqref{2eq:u-04} satisfies
\begin{equation}\label{2eq:u-05}
\vert RHS\vert\lesssim \tau^{-\frac{13}{6}}+\tau^{-(2+\alpha)}+(1+\tau)(1+\tau^{-\frac{2}{3}})e^{-\varsigma r_0 \tau}+\mathcal O(\tau^{-1}e^{-\frac{1}{2}\varsigma r_0 \tau}).
\end{equation}
Furthermore, we have the following inequality about the left-hand side of \eqref{2eq:u-04}
\begin{align}
\vert  LHS \vert \gtrsim \frac{\Gamma(2)| \eta_1\partial_1u_1^{-}(\mathbf 0)|}{\tau^2}\left \vert \frac{\cos\vartheta_M}{((\mathbf d+\mathrm i\mathbf d)\cdot \mathbf {\hat{x}}_1)^2}+\frac{\cos\vartheta_m}{((\mathbf d+\mathrm i\mathbf d)\cdot \mathbf {\hat{x}}_2)^2} \right\vert  -\mathcal O(\tau^{-1}e^{-\lambda r_0\tau}),\label{2eq:u-06}
\end{align}
where $\mathbf {\hat{x}}_1\in \Gamma_{r_0}^{+}$ and $\mathbf {\hat x}_2\in\Gamma_{r_0}^{-}$.
Let $\mathbf d^{\perp}$ be in the form of \eqref{1-eq:d2}, then we have
\begin{align*}
\left \vert \frac{\cos\vartheta_M}{((\mathbf d+\mathrm i\mathbf d)\cdot \mathbf {\hat{x}}_1)^2}+\frac{\cos\vartheta_m}{((\mathbf d+\mathrm i\mathbf d)\cdot \mathbf {\hat{x}}_2)^2} \right\vert  =\vert \cos \vartheta_Me^{-2\mathrm i\vartheta_m}&+\cos \vartheta_me^{-2\mathrm i\vartheta_M}\vert
 \not=0
\end{align*}
for $-\pi<\vartheta_m<\vartheta_M<\pi,\ \vartheta_M-\vartheta_m\in(0,\pi)$. Substituting  \eqref{2eq:u-05}  and \eqref{2eq:u-06}  into \eqref{2eq:u-04}, then multiplying $\tau^2$ on both sides and taking $\tau \to \infty$, we get
$$\vert \eta_1 \partial_1u_1^{-}(\mathbf 0) \vert =0,$$
which implies that $\partial_1u_1^{-}(\mathbf 0) =0$ by noting $\eta_1\not=0$.
Hence, we obtain a contradiction.

\medskip

\noindent\textbf{Case (c)}:  $\partial_2u_1^{-}(\mathbf 0)\not=0,$ $u_1^{-}(\mathbf 0)=\partial_1u_1^{-}(\mathbf 0)=0$. Similar to \eqref{2eq:u-0}, it holds that
\begin{align}
\eta_1\partial_2u_1^{-}(\mathbf 0)&\int_{\Gamma_{r_0}^\pm} e^{\rho\cdot \mathbf x}x_2\mathrm d \sigma= -\eta_1\partial_2u_1^{-}(\mathbf 0)\int_{\Gamma_{r_0}^\pm}\psi(\mathbf x)e^{\rho\cdot \mathbf x}x_2\mathrm d\sigma \notag\\
&+\int_{\mathcal S_{r_0}}(1-q_1)u_1^{-}u_0\mathrm d\mathbf x-\eta_1\int_{\Gamma_{r_0}^\pm}\delta u_1^{-}u_0\mathrm d\sigma+I_4\label{2eq:u3-01}.
\end{align}
Similar to \eqref{2eq:u-04}, for the right-hand and left-hand sides of \eqref{2eq:u3-01}, we have
\begin{equation}\notag
\vert RHS\vert\lesssim \tau^{-\frac{13}{6}}+\tau^{-(2+\alpha )}+(1+\tau)(1+\tau^{-\frac{2}{3}})e^{-\varsigma r_0 \tau}+\mathcal O(\tau^{-1}e^{-\frac{1}{2}\varsigma r_0 \tau})
\end{equation}
and
\begin{equation}\label{2eq:u3-03}
\vert LHS\vert\gtrsim\frac{\Gamma(2)\vert \eta_1\partial_2u_1^{-}(\mathbf 0)\vert}{\tau^2}\left\vert\frac{\sin\vartheta_M}{((\mathbf d+\mathrm i\mathbf d^\perp)\cdot \mathbf {\hat x}_1)^2} + \frac{\sin\vartheta_m}{((\mathbf d+\mathrm i\mathbf d^\perp)\cdot \mathbf {\hat x}_2)^2}  \right\vert-\mathcal O(\tau^{-1}e^{-\lambda r_0\tau})
\end{equation}
where $\mathbf {\hat x}_1\in\Gamma_{r_0}^{+},\ \mathbf {\hat x}_2\in \Gamma_{r_0}^{-}$.
Let $\mathbf d^\perp$ be defined as in \eqref{1-eq:d2}, then it follows that:
$$\left\vert\frac{\sin\vartheta_M}{((\mathbf d+\mathrm i\mathbf d^\perp)\cdot \mathbf {\hat x}_1)^2} + \frac{\sin\vartheta_m}{((\mathbf d+\mathrm i\mathbf d^\perp)\cdot \mathbf {\hat x}_2)^2}  \right\vert=\vert\sin\vartheta_Me^{-2\mathrm i\vartheta_m}+\sin\vartheta_me^{-2\mathrm i\vartheta_M} \vert \not=0,$$
for $-\pi<\vartheta_m<\vartheta_M<\pi,\ \vartheta_M-\vartheta_m\in(0,\pi)$ and $\vartheta_M-\vartheta_m\neq\frac{\pi}{2}.$
Combining \eqref{2eq:u3-01}-\eqref{2eq:u3-03}and taking $\tau\to \infty$, one has
$$
\vert \eta_1\partial_2u_1^{-}(\mathbf 0)\vert=0,
$$
Hence, it yields that $\partial_2u_1^{-}(\mathbf 0)=0$ since $\eta_1\neq0$, which contradicts to $\partial_2u_1^{-}(\mathbf 0)\not=0$.

\medskip

\noindent \textbf{Case (d)}: $\partial_1u_1^{-}(\mathbf 0)\not=0$,$\partial_2u_1^{-}(\mathbf 0)\not=0$, $u_1^{-}(\mathbf 0)=0$.  Similar to \eqref{2eq:u-0}, we have
\begin{align}
\eta_1\int_{\Gamma_{r_0}^\pm}&(\partial_1u_1^{-}(\mathbf 0)x_1+\partial_2u_1^{-}(\mathbf 0)x_2)e^{\rho\cdot \mathbf x}\mathrm d\sigma=\int_{\mathcal S_{r_0}}(1-q_1)u_1^{-}u_0\mathrm d\sigma-\eta_1\int_{\Gamma_{r_0}^{\pm}}\delta u_1^{-}u_0\mathrm d\sigma
\notag\\
&-\eta_1\int_{\Gamma_{r_0}^\pm} (\partial_1u_1^{-}(\mathbf 0)x_1+\partial_2u_1^{-}(\mathbf 0)x_2)e^{\rho\cdot \mathbf x}\psi(\mathbf x)\mathrm d\sigma
+I_4. \label{eq:case d 415}
\end{align}
By careful calculations, for the right-hand and left-hand sides of \eqref{eq:case d 415}, it is evident that
\begin{equation}\notag
\vert RHS\vert \lesssim \tau^{-\frac{13}{6}}+\tau^{-(\alpha +2)}+(1+\tau)(1+\tau^{-\frac{2}{3}})e^{-\varsigma r_0\tau}+\mathcal O(\tau^{-\frac{7}{6}}e^{-\frac{1}{2}\varsigma r_0\tau}),
\end{equation}
and
\begin{equation}\notag
\vert LHS\vert \gtrsim \frac{\Gamma(2)\vert \eta_1\vert \vert \mathcal N\vert}{\tau^2\vert (\mathbf d+\mathrm i\mathbf d^\perp)\cdot \mathbf {\hat x}_1\vert^2\vert(\mathbf d+\mathrm i\mathbf d^\perp)\cdot \mathbf {\hat x}_2\vert^2}-\mathcal O(\tau^{-1}e^{-\lambda r_0\tau}).
\end{equation}
Similarly, it yields that
\begin{equation}\notag
\vert \mathcal N\vert =0
\end{equation}
as $\tau\to \infty$.
When $\mathbf d^\perp$ is given in the form of \eqref{1-eq:d1}, we have
\begin{align}\notag
\mathcal N=&(\partial_1 u_1^{-}(\mathbf 0)\cos\vartheta_M+\partial_2 u_1^{-}(\mathbf 0)\sin\vartheta_M)e^{2\mathrm i\vartheta_m}\notag\\
&\quad+ (\partial_1 u_1^{-}(\mathbf 0)\cos\vartheta_m+\partial_2 u_1^{-}(\mathbf 0)\sin\vartheta_m)e^{2\mathrm i\vartheta_M},\label{2eq:N}
\end{align}
 And when $\mathbf d^\perp$ is given by \eqref{1-eq:d2},  one has
\begin{align}\notag
\mathcal N=&(\partial_1 u_1^{-}(\mathbf 0)\cos\vartheta_M+\partial_2 u_1^{-}(\mathbf 0)\sin\vartheta_M)e^{-2\mathrm i\vartheta_m}\notag\\
&\quad+ (\partial_1 u_1^{-}(\mathbf 0)\cos\vartheta_m+\partial_2 u_1^{-}(\mathbf 0)\sin\vartheta_m)e^{-2\mathrm i\vartheta_M}.\label{2eq:N2}
\end{align}



Let us consider the real and imaginary parts of $\partial_1u_1^{-}(\mathbf 0)$ and  $\partial_2u_1^{-}(\mathbf 0)$, seperately. Denote $\partial_1u_1^{-}(\mathbf 0)=a_1+b_1\mathrm i$ and $ \partial_2u_1^{-}(\mathbf 0)=a_2+b_2\mathrm i$, where  $a_i,b_i\in \mathbb R\ (i=1,2).$ For \eqref{2eq:N}, it follows that
\begin{align}
\Re(\mathcal N)&=a_1(\cos\vartheta_M\cos2\vartheta_m+\cos \vartheta_m\cos 2\vartheta_M)-b_1(\cos\vartheta_M\sin2\vartheta_m+\cos \vartheta_m\sin2\vartheta_M)\notag\\
&+a_2(\sin\vartheta_M\cos 2\vartheta_m+ \sin\vartheta_m\cos 2\vartheta_M)
-b_2(\sin \vartheta_M\sin 2\vartheta_m+\sin \vartheta_m\sin 2\vartheta_M),\label{eq:N01R}\\
\Im(\mathcal N)&=a_1(\cos\vartheta_M\sin 2\vartheta_m+\cos \vartheta_m\sin 2\vartheta_M)+b_1(\cos\vartheta_M\cos2\vartheta_m+\cos \vartheta_m\cos2\vartheta_M)\notag\\
&+a_2(\sin\vartheta_m\sin 2\vartheta_M+ \sin\vartheta_m\sin 2\vartheta_M)+b_2(\sin \vartheta_M\cos 2\vartheta_m+\sin \vartheta_m\cos 2\vartheta_M).\notag
\end{align}
Similarly,  for \eqref{2eq:N2}, one has
\begin{align}
\Re(\mathcal N)&=a_1(\cos\vartheta_M\cos 2\vartheta_m+\cos \vartheta_m\cos 2\vartheta_M)+b_1(\cos\vartheta_M\sin2\vartheta_m+\cos \vartheta_m\sin2\vartheta_M)\notag\\
&+a_2(\sin\vartheta_M\cos 2\vartheta_m+ \sin\vartheta_m\cos 2\vartheta_M)+b_2(\sin \vartheta_M\sin 2\vartheta_m+\sin \vartheta_m\sin 2\vartheta_M),\notag\\
\Im(\mathcal N)&=-a_1(\cos\vartheta_M\sin 2\vartheta_m+\cos \vartheta_m\sin 2\vartheta_M)+b_1(\cos\vartheta_M\cos2\vartheta_m+\cos \vartheta_m\cos2\vartheta_M)\notag\\
&-a_2(\sin\vartheta_m\sin 2\vartheta_M+ \sin\vartheta_m\sin 2\vartheta_M)+b_2(\sin \vartheta_M\cos 2\vartheta_m+\sin \vartheta_m\cos 2\vartheta_M).\label{eq:N02I}
\end{align}
Combining \eqref{eq:N01R}-\eqref{eq:N02I} with the fact that $|\mathcal N|=0$, we can derive that
\begin{equation}\label{eq:ab=0}
\mathcal M_1\left(
\begin{matrix}
a_1\\
a_2
\end{matrix}
\right)=0, \
\mathcal M_2\left(
\begin{matrix}
b_1\\
b_2
\end{matrix}
\right)=0,
\end{equation}
where
\begin{equation}\notag
\mathcal M_1=\left(
\begin{matrix}
\cos \vartheta_M\cos2\vartheta_m+ \cos \vartheta_m\cos2\vartheta_M & \sin \vartheta_M\cos 2\vartheta_m+\sin \vartheta_m\cos 2\vartheta_M\\
\cos \vartheta_M\sin2\vartheta_m+ \cos \vartheta_m\sin2\vartheta_M & \sin \vartheta_M\sin 2\vartheta_m+\sin \vartheta_m\sin 2\vartheta_M
\end{matrix}
\right)
\end{equation}
and
\begin{equation}\notag
\mathcal M_2=\left(
\begin{matrix}
\cos \vartheta_M\sin2\vartheta_m+ \cos \vartheta_m\sin2\vartheta_M & \sin \vartheta_M\sin 2\vartheta_m+\sin \vartheta_m\sin 2\vartheta_M\\
\cos \vartheta_M\cos2\vartheta_m+ \cos \vartheta_m\cos2\vartheta_M & \sin \vartheta_M\cos 2\vartheta_m+\sin \vartheta_m\cos 2\vartheta_M
\end{matrix}
\right).
\end{equation}
 According to the admissible condition, we know that $\vartheta_M-\vartheta_m\in (0,\pi)$ and $\vartheta_M-\vartheta_m\not=\frac{\pi}{2}$. After tedious  calculations, we have
$$
\det(\mathcal M_1)=-\det(\mathcal M_2)=-2\sin^2(\vartheta_M-\vartheta_m)\cos(\vartheta_M-\vartheta_m)\neq 0.
 $$
Therefore, by \eqref{eq:ab=0}, we have $a_i=b_i=0$ for $i=1,2$, which indicates that $\partial_1u_1^{-}(\mathbf 0)=\partial_2u_1^{-}(\mathbf 0)=0$. This leads to a contradiction.

The proof is complete.
\end{proof}

\begin{rem}
From the proof of Theorem \ref{2thm:un-shape}, we can see that when the conductive boundary parameter $\eta$ is a H\"older continuous function on the boundary and $\eta$  does not vanish at the vertex of the scatterer, we can also establish the local uniqueness result for determining the shape of an admissible conductive  polygonal-nest or polygonal-cell medium scatterer by a single far-field pattern. The corresponding proof can be modified slightly from the proof of Theorem \ref{2thm:un-shape}.  Similarly, a global unique result for identifying the shape of an admissible conductive polygonal-nest medium scatterer by a single  far-field pattern can be obtained for a H\"older continuous conductive boundary parameter  with the non-vanishing property at the vertex.
\end{rem}

Considering an admissible conductive polygonal-cell medium scatterer, the following theorem states the unique identifiability for the refractive index and conductive boundary parameter by a single far-field measurement, where we require a-priori knowledge of the shape and  structure of the admissible conductive polygonal-cell medium scatterer.

\begin{thm}\label{2thm:un-eta-cell}
Suppose that  $(\Omega;q_\ell,\eta_\ell),\ell=1,2$ are two admissible conductive polygon-cell medium scatterers associated with the scattering problem \eqref{eq:sca}.  Suppose that the refractive index $q_\ell$ and conductive boundary parameter $\eta_\ell$ have a common polygonal-cell structure $\cup_{i=1}^N\Sigma_{i}$ described by Definition \ref{1-def:cell}. Let  $q_\ell$ and $\eta_\ell$ be defined as follows:
$$
q_\ell=\sum_{i=1}^Nq_{i,\ell}\chi_{\Sigma_{i}},\ \eta_\ell=\sum_{i=1}^N\tilde \eta_\ell \chi_{\Sigma_{i}}.
$$
where the coefficients of $q_{i,\ell}$ are complex valued following  the form specified in \eqref{2eq:qform} and $\tilde \eta_\ell$ {$\in \mathbb C$}. Let $u_\ell ^\infty(\mathbf {\hat x};u^i)$ be the far-field pattern associated with the incident wave $u^i$ corresponding to $(\Omega;q_\ell,\eta_\ell)$. Assume that
$$
 u_1 ^\infty(\mathbf {\hat x};u^i)=u_2 ^\infty(\mathbf {\hat x};u^i)
 $$
for all $\mathbf {\hat x} \in \mathbb S^{1}$ and a fixed incident wave $u^i$. Then we have $q_1=q_2$ and $\eta_1=\eta_2.$
\end{thm}

Considering an admissible conductive polygonal-nest medium scatterer,  we have shown that the shape of the underlying medium scatterer can be uniquely determined by a single far-field measurement in Theorem \ref{2thm:un-shape}. In the following theorem, we will further prove that the polygonal-nest structure of of the underlying scatterer, the corresponding reflective  indexes and the conductive boundary parameter can be uniquely determined by a single far-field pattern.

\begin{thm}\label{2thm:un-eta-nest}
Suppose $(\Omega_\ell;q_\ell,\eta_\ell),\ell=1,2$ are two admissible polygonal-nest medium scatterers associated with the scattering problem \eqref{eq:sca}, where $\Omega_1=\bigcup_{i=1}^{N_1}\mathcal D_{i,1}$ and $\Omega_2=\bigcup_{i=1}^{N_2}\mathcal D_{i,2}$ with $\mathcal D_{i,\ell}=\Sigma_{i,\ell}\setminus \overline{\Sigma_{i,\ell}},i=1,\dots,N_\ell$. For each $\ell$, let $q_\ell$ and $\eta_\ell$ be defined as follows:
$$
q_\ell=\sum_{i=1}^{N_\ell}q_{i,\ell}\chi_{\mathcal D_{i,\ell}},\ \eta_\ell=\sum_{i=1}^{N_\ell} \eta _{i,\ell}\chi_{\partial \Sigma_{i,\ell}}.
$$
The coefficients of $q_{i,\ell}$ are complex valued, where $q_{i,\ell}$ follows the form specified in \eqref{2eq:qform}, while $\eta _{i,\ell}$ {$\in \mathbb C$}. Let $u_\ell ^\infty(\mathbf {\hat x};u^i)$ be the far-field pattern associated with the incident wave $u^i$ corresponding to $(\Omega_\ell;q_\ell,\eta_\ell)$. Assume that
$$  u_1 ^\infty(\mathbf {\hat x};u^i)=u_2 ^\infty(\mathbf {\hat x};u^i)$$
for all $\mathbf {\hat x} \in \mathbb S^{1}$ and a fixed incident wave $u^i$. Then we have $N_1=N_2=N$, $\partial \Sigma_{i,1}=\partial \Sigma_{i,2}$, $q_{i,1}=q_{i,2}$ and $\eta_{i,1}=\eta_{i,2}$ for $i=1,\dots, N$.
\end{thm}

The proofs of Theorems \ref{2thm:un-eta-cell} and \ref{2thm:un-eta-nest} are deferred to Section \ref{sec:4}. Prior to that, we require the following lemmas.
Lemma \ref{lem:holdernew} establishes that the solution $u^- \in H^{1}(\Omega)$ to the conductive scattering problem \eqref{eq:sca} is H\"older continuous in the vicinity of the corner. Furthermore, Lemma \ref{2thm:conductive} investigates the local behavior of solutions to the coupled PDE system \eqref{2-eq:conductive} near a corner. Furthermore, we obtain local uniqueness for  identifying both the refractive index and the conductive boundary parameter of the coupled PDE system.

 \begin{lem}\label{lem:holdernew}
Consider a conductive medium scatterer  $(\Omega;q,\eta)$ possessing   either a polygonal-nest structure or  polygonal-cell structure defined in Definitions \ref{1-def:pi-nest} and  \ref{1-def:pi-cell} associated with the  convex polygons $\Sigma_i,i=1,\dots,N$ given in Definitions \ref{1-def:nest} or \ref{1-def:cell}.
Recall that  $\mathcal V(\Sigma_i)$  and  $\mathcal V(\Omega)$  as defined in Definition \ref{2def:adm} denote  the vertex sets of the polygons $\Sigma_i$ for   the polygonal-nest and of the $\Omega$ for  the polygonal-cell structure, respectively. 
 Let $\EM x_p\in \mathcal V(\Sigma_i)$  or $\EM x\in \mathcal V(\Omega)$ be  a vertex point. For sufficiently small radius $r_0$, define $\mathcal S_{\EM x_p,r_0}=B_{\EM x_p,r_0}\cap \Sigma_i$ for  the polygonal-nest structure and   $\mathcal S_{\EM x_p,\rho}=B_{\EM x_p,r_0}\cap \Omega$  for the polygonal-cell structure. If $u=u^{-}\chi_{\Omega}+u^{+}\chi_{\mathbb R^2\setminus \overline{\Omega}}\in H^1_{loc}(\mathbb R^2)$ satisfies  the conductive scattering problem \eqref{eq:sca}, then one has $u^{-}\in C^\alpha(\overline{\mathcal S_{\EM x_p,r_0}}),\ \alpha\in(0,1]$.
\end{lem}

\begin{figure}
\centering
\includegraphics[width=5cm]{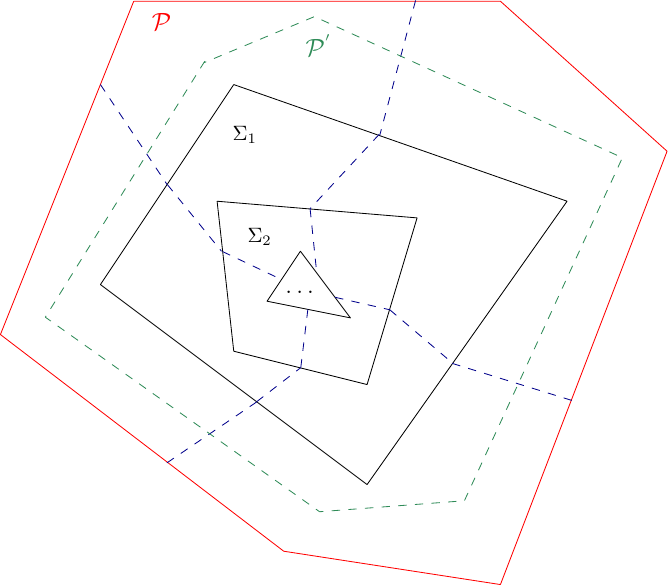}
\caption{A schematic illustration of network for polygonal-nest structure.}
\label{fig:network}
\end{figure}

\begin{proof}
Let $\mathcal P$ and $\mathcal P^{'}$ be   bounded, connected and convex polygons such that $\overline{\Omega}\subset \overline{\mathcal P{'}}\subset \overline{\mathcal P}$. Here, $\Omega$   has either a polygonal-nest structure or  polygonal-cell structure. First, by an argument analogous to the proof   in Lemma   \ref{2lem:reggularity}, let $\varphi\in C^{\infty}(\mathcal P)$ be a smooth cutoff function satisfying
$$
\varphi\equiv1 \mbox{ in } \Omega,\quad \varphi\equiv 0 \mbox{ in }\mathcal P \setminus  \mathcal P^{'}.
$$
Define $\tilde {\varphi}=u\varphi$. Using the conductive scattering problem \eqref{eq:sca}, we obtain that
\begin{align}
        \Delta\tilde{\varphi}=f\ \mbox{in}\ \Omega,\quad
        \Delta\tilde{\varphi}=g\ \mbox{in}\ \mathcal P\setminus{\overline{\Omega}},\quad
        \tilde{\varphi}=0\  \mbox{on}\ \partial \mathcal P.\notag
\end{align}
where $f=(-k^2qu^{-})\varphi+\Delta\varphi u^{-}+2\nabla u^{-}\cdot \nabla\varphi$ and $g=(-k^2u^{+})\varphi+\Delta\varphi u^{+}+2\nabla u^{+}\cdot \nabla\varphi$.  Recall that $u^{-}\in H^1({\Omega})$ and $u^{+}\in H^1(\mathbb R^2\setminus \overline{\Omega})$, it is clear that $f\in L^2(\Omega)$ and $g\in L^2(\mathcal P\setminus{\overline{\Omega}})$.
Moreover, through appropriate segmentation of $\mathcal P$ and $\Omega$,  these domains naturally constitute a polygonal network as defined in \cite[Definition 8.A]{DN}.
The polygonal-cell structure case follows immediately. Figure \ref{fig:network} provides a schematic illustration of the polygonal-nest structure case.  Moreover,  it is noted that $\tilde {\varphi}$ satisfies the continuity condition at the interface within the interior of $\mathcal P$.  Consequently, by  \cite[Theorem 10.28]{DN}, we conclude that  $\tilde{\varphi}\in H^2(\mathcal S_{\EM x_p,\rho})$ due to the convexity of $\mathcal S_{\EM x_p,\rho}$. Therefore, by virtue of $\varphi\equiv 1$ in $\Omega$,   one has $u^-\in H^2(\mathcal S_{\EM x_p,\rho})$. Thus,  the Sobolev embedding theorem  implies that $u^{-}\in C^{\alpha}(\overline{\mathcal S_{\EM x_p,\rho}}),\alpha\in (0,1]$.

The proof is complete.
\end{proof}


\begin{lem}\label{2thm:conductive}
Recall that $\mathcal S_{r_0},\Gamma_{r_0}^\pm$ and $B_{r_0}$ are defined in \eqref{1eq:sr0}.
Suppose that $u_1=u_1^{-}\chi_{\mathcal S_{r_0}}+u_1^{+}\chi_{B_{r_0}\setminus\overline{\mathcal S_{r_0}}}\in H^{1}(B_{r_0})$ and $u_2=u^{-}_2\chi_{\mathcal S_{r_0}}+u^{+}_{2}\chi_{B_{r_0}\setminus\overline{\mathcal S_{r_0}}}\in H^{1}(B_{r_0})$ satisfy  the following PDE system:
\begin{equation}\label{2-eq:conductive}
\begin{cases}
(\Delta+k^2q_{1})u_1^{-}=0,\hspace*{2.75cm} \mbox{in}\ \mathcal S_{r_0},\\
(\Delta+k^2 q_{2})u_2^{-}=0,\hspace*{2.75cm} \mbox{in}\ \mathcal S_{r_0},\\
(\Delta+k^2q_3)u^{+}_1=0,\hspace*{2.75cm} \mbox{in}\ B_{r_0}\setminus\overline{\mathcal S_{r_0}},\\
(\Delta+k^2q_3)u^{+}_2=0,\hspace*{2.75cm} \mbox{in}\ B_{r_0}\setminus\overline{\mathcal S_{r_0}},\\
u^{+}_1=u^{-}_1,\ \partial_{\nu}u^{-}_1=\partial_{\nu}u^{+}_1+\eta_1u^{+}_1,\hspace*{0.45cm} \mbox{on}\ \Gamma_{r_0}^\pm,\\
u^{+}_2=u^{-}_2,\ \partial_{\nu}u^{-}_2=\partial_{\nu}u^{+}_2+\eta_2 u^{+}_2,\hspace*{0.45cm} \mbox{on}\ \Gamma_{r_0}^\pm,\\
\end{cases}
\end{equation}
where $\eta_1$, $\eta_{2}$ {$ \in \mathbb C$} and $ q_1,\  q_{2},\ q_{3}$ with forms
\begin{align*}
 q_j=q_j^{(0)}+q_j^{(1)} x_1+q_j^{(2)} x_2,\quad j=1,2,3, 
\end{align*}
where $q_j^{(0)}$, $q_j^{(1)}$ and $q_j^{(2)}$ {$\in \mathbb C$}, $j=1,2,3$.

{Assume that
\begin{equation}\label{2eq:admissible le33 1}
\lim_{\rho\to 0}\frac{1}{m( B(\mathbf x_p,\rho))}\int_{B(\mathbf x_p,\rho)} \vert u_i(\mathbf{x})\vert \mathrm d\mathbf x\neq 0,\ i=1\ \mbox{or}\ 2.
\end{equation}}
If
 $$
 u^{+}_1=u^{+}_2\ \mbox{in}\ B_{r_0}\setminus\overline{\mathcal S_{r_0}},
 $$
then  one has $\eta_1=\eta_2,\ q_1=q_{2}.$ Moreover, it holds that $u_1^{-}=u_2^{-}$ in $\mathcal S_{r_0}$.
\end{lem}

\begin{proof}

Since $u^{+}_1=u_{2}^{+}$ in $B_{r_0}\setminus \overline{\mathcal S_{r_0}}$,  we have that $u^{+}_1=u^{+}_2\ \mbox{on}\ \Gamma_{r_0}^\pm$ by the trace theorem, which indicates that $u_1^{-}=u_1^{+}=u_2^{+}=u_2^{-}\mbox{on}\ \Gamma_{r_0}^\pm$ according to the transmission boundary condition. Without loss of generality, we assume that $u_2$ fulfills \eqref{2eq:admissible le33 1}. Denote $\tilde u:=u^{-}_{1}-u_2^{-}$, then it can be shown directly that
\begin{equation}\notag
\begin{cases}
\Delta\tilde u+k^2q_1\tilde u=k^2(q_2-q_1)u_2^{-},\ \hspace*{0.3cm}\mbox{in}\ \mathcal S_{r_0},\\
\tilde u=0,\ \partial_{\nu}\tilde u=(\eta_1-\eta_2)u_2^{-}\hspace*{1.1cm}\mbox{on}\ \Gamma_{r_0}^\pm.
\end{cases}
\end{equation}

Similar to the proof of Theorem \ref{2thm:un-shape}, we know that there exists the CGO solution $u_0$ given by \eqref{1-eq:cgo}, which satisfies \eqref{eq:cgo par} and
\begin{equation}\notag 
(\Delta+k^2 q_1)u_0=0\mbox{ in } \mathcal S_{r_0}.
\end{equation}
Using Green's formula, we have the following integral identity
\begin{align}\label{2eq:greenid lemm33}
\int_{\mathcal S_{r_0}}k^2(q_2-q_1)u_2^{-}u_{0}\mathrm d\mathbf x=(\eta_1-\eta_2)\int_{\Gamma_{r_0}^\pm}u_2^{-}u_0\mathrm d\sigma+\int_{\Lambda_{r_0}}u_0\partial_{\nu}\tilde u-\tilde u\partial_{\nu}u_0\mathrm d\sigma.
\end{align}

{By virtue of Lemma \ref{lem:holdernew}, we know that $u_i^{-}\in C^\alpha(\overline{\mathcal S_{r_0}}),\alpha\in (0,1],i=1,2.$
}
Thus,  $u_2^{-}$ satisfies the following expansion:
\begin{align}\label{eq:u2- ex}
u_2^{-}&=u_2^{-}(\mathbf 0) +\delta u_2^{-},\ \vert \delta u_2^{-}\vert \leq \vert \mathbf x\vert ^{\alpha }\|u_2^{-}\|_{C^{\alpha }}.
\end{align}
Now, we are going to prove that $\eta_1=\eta_2$ and $q_1=q_2$ in two parts.

\medskip
\noindent\textbf{Part 1: $\eta_1=\eta_2$}.
Using \eqref{2eq:greenid lemm33} and \eqref{eq:u2- ex}, we can directly  see that
\begin{equation}\label{2eq:case1eta}
(\eta_2-\eta_1)u_2^{-}(\mathbf 0)\int_{\Gamma_{r_0}^\pm}e^{\rho\cdot \mathbf x}\mathrm d\sigma=\sum_{i=1}^{4}I_i
\end{equation}
where
\begin{align}
I_1&=(\eta_1-\eta_2)u_2^{-}(\mathbf 0)\int_{\Gamma_{r_0}^\pm}\psi(\mathbf x)e^{\rho\cdot \mathbf x}\mathrm d\sigma,
I_2=\int_{\mathcal S_{r_0}}k^2(q_1-q_2)u_2^{-}u_0\mathrm d\mathbf x,\notag\\
I_3&=(\eta_1-\eta_2)\int_{\Gamma_{r_0}^\pm} \delta u_2^{-}u_0(\mathbf x)\mathrm d\sigma,
I_4 =\int_{\Lambda_{r_0}}u_0\partial_{\nu}\tilde u-\tilde u \partial_\nu u_0\mathrm d\sigma.\notag
\end{align}
By virtue of \eqref{2eq:u-01} and \eqref{2eq:u-02}, we can derive that the right-hand side of \eqref{2eq:case1eta} satisfies
$$\vert RHS \vert \lesssim \tau^{-\frac{7}{6}}+\tau^{-(2+\alpha)}+\mathcal O(\tau^{-1}e^{-\frac{1}{2}\varsigma r_0\tau})+(1+\tau)(1+\tau^{-\frac{2}{3}})e^{-\varsigma r_0\tau}.$$

By \eqref{1-eq:d2} and \eqref{2eq:u-03}, one has the left-hand side of \eqref{2eq:case1eta} satisfying
\begin{align}
\vert LHS\vert  \gtrsim
\vert (\eta_2-\eta_1)u_2^{-}(\mathbf 0)\vert \frac{\Gamma(1)\vert e^{\mathrm i\vartheta_m}+e^{\mathrm i\vartheta_M}\vert }{ \vert \tau e^{\mathrm i\varphi} \vert }-\mathcal O(\tau^{-1}e^{-\frac{1}{2}\lambda r_0 \tau}).\label{2eq:I}
\end{align}
Similar to the proof of \textbf{Case(a)} in Theorem \ref{2thm:un-shape}, we have
$$|(\eta_1-\eta_2)u_2^{-}(\mathbf 0)|=0,\ \mbox{for} \ \vartheta_M-\vartheta_m\in(0,\pi),$$
which indicates that $\eta_1=\eta_2$ since  $u_2^{-}(\mathbf 0)\neq0. $

\medskip
\noindent\textbf{Part 2: $q_1=q_2$}.  Using $\eta_1=\eta_2$, the integral identity \eqref{2eq:case1eta} can be rewritten as
\begin{equation}\label{eq:int lem3 q0}
k^2(q_2^{(0)}-q_1^{(0)} )u_2^{-}(\mathbf 0)\int_{\mathcal S_{r_0}}e^{\rho\cdot \mathbf x}\mathrm d\mathbf x=\sum_{i=1}^4 \mathcal P_i+I_4,
\end{equation}
where
\begin{align}
\mathcal P_1&=k^2(q_1^{(0)}-q_2^{(0)})u_2^{-}(\mathbf 0)\int_{\mathcal S_{r_0}}\psi(\mathbf x)e^{\rho\cdot \mathbf x}\mathrm d\mathbf x,\
\mathcal P_2=k^2(q_1^{(0)}-q_2^{(0)}) \int_{\mathcal S_{r_0}} \delta u_2^{-}u_0\mathrm d\mathbf x,\notag\\
\mathcal P_3&=k^2(q_1^{(1)}-q_{2}^{(1)})u_2^{-}(\mathbf 0)\int_{\mathcal S_{r_0} }x_1u_0\mathrm d\mathbf x+k^2(q_1^{(2)}-q_2^{(2)})u_2^{-}(\mathbf 0)\int_{\mathcal S_{r_0} }x_2u_0\mathrm d\mathbf x,\notag\\
\mathcal P_4&=k^2(q_1^{(1)}-q_2^{(1)})u_2^{-}(\mathbf 0)\int_{\mathcal S_{r_0} }x_1  \delta u_2^{-}u_0\mathrm d\mathbf x+k^2(q_1^{(2)}-q_{2}^{(2)})u_2^{-}(\mathbf 0)\int_{\mathcal S_{r_0} }x_2   \delta u_2^{-}u_0\mathrm d\mathbf x,\notag
\end{align}
and $I_4$ is defined in \eqref{2eq:case1eta}.
Then by Proposition \ref{1-prop:some est}, it is directly shown that
\begin{align}
|\mathcal P_1|&\lesssim \tau^{-29/12}+\mathcal O(\tau^{-1}e^{-1/2\varsigma r_0\tau}),\ |\mathcal P_2|\lesssim \tau^{-3}+\tau^{-(\alpha+3)}+\mathcal O(\tau^{-1}e^{-1/2\varsigma r_0\tau}),\notag\\
|\mathcal P_3|&\lesssim  \int_{\mathcal S_{r_0}} |x|| u_0|\mathrm d\mathbf x \lesssim \tau^{-3}+\mathcal O(\tau^{-1}e^{-1/2\varsigma r_0\tau}),\label{2eq:q01}\\
|\mathcal P_4|&\lesssim  \int_{\mathcal S_{r_0}}\vert x\vert \vert \tilde \delta u_2^{-}u_0\vert \mathrm d\mathbf x
\lesssim \tau^{-4}+\tau^{-(\alpha+4)}+O(\tau^{-1}e^{-1/2\varsigma r_0\tau}).\notag
\end{align}

Let $\mathbf d^\perp$ be in the form of \eqref{1-eq:d2}, we find that the left-hand side of \eqref{eq:int lem3 q0} satisfies
\begin{equation}\label{2eq:q0}
 LHS=k^2(q_0^{(2)}-q_0^{(1)})u_2^{-}(\mathbf 0)\frac{\Gamma(2)(e^{2\mathrm i\vartheta_M}-e^{2\mathrm i\vartheta_m})}{\tau^2e^{2\mathrm i\varphi}}+\mathcal O(\tau^{-1}e^{-\frac{1}{2}\lambda r_0\tau}).
 \end{equation}
Similarly, by combining \eqref{eq:int lem3 q0}-\eqref{2eq:q0}, as $\tau\to \infty$, we can obtain
$$\vert k^2(q_0^{(2)}-q_0^{(1)})u_2^{-}(\mathbf 0)\vert =0,$$
since $\vert e^{2\mathrm i\vartheta_M}-e^{2\mathrm i\vartheta_m}\vert \neq 0$ when $\vartheta_M-\vartheta_m\in (0,\pi)$. Therefore, $q_0^{(2)}=q_0^{(1)}$ holds due to $ u_2^{-}(\mathbf 0)\neq 0$.

Using $q_1^{(0)}=q_2^{(0)}$, from \eqref{eq:int lem3 q0}, we have the following integral identity
\begin{align}
k^2(q_2^{(1)}-q_1^{(1)})u_2^{-}(\mathbf 0)\int_{\mathcal S_{r_0}}x_1e^{\rho\cdot \mathbf x}\mathrm d\mathbf x
+
k^2(q_2^{(2)}-q_1^{(2)})u_2^{-}(\mathbf 0)\int_{\mathcal S_{r_0}}x_2e^{\rho\cdot \mathbf x}\mathrm d\mathbf x=\sum_{i=1}^2\mathcal E_i+I_4,\label{2eq:q1q2}
\end{align}
where
\begin{align}
\mathcal E_1&=k^2(q_1^{(1)}-q_2^{(1)})u_2^{-}(\mathbf 0)\int_{\mathcal S_{r_0}}x_1\psi(\mathbf x)e^{\rho\cdot \mathbf x}\mathrm d\mathbf x
+k^2(q_1^{(2)}-q_2^{(2)})u_2^{-}(\mathbf 0)\int_{\mathcal S_{r_0}}x_2\psi(\mathbf x)e^{\rho\cdot \mathbf x}\mathrm d\mathbf x,\notag\\
\mathcal E_2&=k^2(q_1^{(1)}-q_2^{(1)})\int_{\mathcal S_{r_0}}x_1 \delta u_2^{-}u_0\mathrm d\mathbf x
+k^2(q_1^{(2)} -q_2^{(2)})\int_{\mathcal S_{r_0}}x_2  \delta u_2^{-}u_0\mathrm d\mathbf x.\notag
\end{align}
Utilizing  Proposition \ref{1-prop:some est}, it can be calculated  estimates about $\mathcal E_1$ and $\mathcal E_2$ as follows,
\begin{align}
\vert \mathcal E_1\vert &\lesssim\int_{\mathcal S_{r_0}}\vert \mathbf x\vert \vert  \psi(\mathbf x) e^{\rho\cdot \mathbf x}\vert \mathrm d\mathbf x
\lesssim \tau^{-\frac{41}{12}}+\mathcal O(\tau^{-1}e^{-\frac{1}{2}\varsigma r_0\tau}),\notag\\
\vert \mathcal E_2\vert &\lesssim \int_{\mathcal S_{r_0}} \vert \mathbf x\vert |  \delta u_2^{-}|\vert u_0\vert \mathrm d\mathbf x
\lesssim \tau^{-(2+\alpha)}+\tau^{-(\frac{41}{12}+\alpha)}+\mathcal O(\tau^{-1}e^{-\frac{1}{2}\varsigma r_0\tau}).\label{2eq:este1}
\end{align}
 Then we can write the left-hand side of \eqref{2eq:q1q2} as
 \begin{align}
 \overline {q}_1\int_{\vartheta_m}^{\vartheta_M}\mathrm d\vartheta\int_{0}^{r_0}r^2\cos\vartheta e^{-\tau e^{\mathrm i(\vartheta-\varphi)r}}\mathrm dr+\overline q_2 \int_{\vartheta_m}^{\vartheta_M}\mathrm d\vartheta\int_{0}^{r_0}r^2\sin\vartheta e^{-\tau e^{\mathrm i(\vartheta-\varphi)r}}\mathrm dr,\label{2eq:d1q}
 \end{align}
 when $\mathbf d^\perp$ is selected in the form described in \eqref{1-eq:d1}, and
  \begin{align}
 \overline q_1\int_{\vartheta_m}^{\vartheta_M}\mathrm d\vartheta\int_{0}^{r_0}r^2\cos\vartheta e^{-\tau e^{\mathrm i(\varphi-\vartheta)r}}\mathrm dr+\overline q_2 \int_{\vartheta_m}^{\vartheta_M}\mathrm d\vartheta\int_{0}^{r_0}r^2\sin\vartheta e^{-\tau e^{\mathrm i(\varphi-\vartheta)r}}\mathrm dr,\label{2eq:d2q}
 \end{align}
  when $\mathbf d^\perp$ is selected in the form of \eqref{1-eq:d2}, where $\overline q_i:=k^2(q_2^{(i)}-q_1^{(i)} )u_2^{-}(\mathbf 0),i=1,2$.

Utilizing \eqref{1-eq:some est1}, the first term of \eqref{2eq:d1q} can be expressed as:
\begin{align}
\frac{\Gamma(3)\overline {q}_1e^{3\mathrm i\varphi}}{\tau^3}&\int_{\vartheta_m}^{\vartheta_M}\cos\vartheta e^{-3\mathrm i\vartheta}\mathrm d\vartheta
+\mathcal O(\tau^{-1}e^{-\frac{1}{2}\varsigma r_0\tau})\notag\\
&=\frac{\Gamma(3)e^{3\mathrm i\varphi}}{\tau^3}\cdot \frac{1}{8}\left\{(-\sin \vartheta+3\mathrm i\cos\vartheta)e^{-3\mathrm i\vartheta} \right\}\bigg \vert_{\vartheta_m}^{\vartheta_M}+\mathcal O(\tau^{-1}e^{-\frac{1}{2}\lambda r_0\tau})\label{2eq:L1}.
\end{align}
The second term can be simplified as:
\begin{align}
\frac{\Gamma(3)\overline {q}_2e^{3\mathrm i\varphi}}{\tau^3}&\int_{\vartheta_m}^{\vartheta_M}\sin\vartheta e^{-3\mathrm i\vartheta}\mathrm d\vartheta
+\mathcal O(\tau^{-1}e^{-\frac{1}{2}\varsigma r_0\tau})\notag\\
&=\frac{\Gamma(3)e^{3\mathrm i\varphi}}{\tau^3}\cdot \frac{1}{8} \left\{ (\cos\vartheta+3\mathrm i\sin\vartheta)e^{-3\mathrm i\vartheta}\right\}  \bigg \vert_{\vartheta_m}^{\vartheta_M}+\mathcal O(\tau^{-1}e^{-\frac{1}{2}\lambda r_0\tau})\label{2eq:L2}.
\end{align}
Similarly, the first term of \eqref{2eq:d2q} can be expressed as:
\begin{align}
\frac{\Gamma(3)\overline {q}_1e^{-3\mathrm i\varphi}}{\tau^3}&\int_{\vartheta_m}^{\vartheta_M}\cos\vartheta e^{3\mathrm i\vartheta}\mathrm d\vartheta
+\mathcal O(\tau^{-1}e^{-\frac{1}{2}\varsigma r_0\tau})\notag\\
&=\frac{\Gamma(3)e^{-3\mathrm i\varphi}}{\tau^3}\cdot \frac{1}{8}\left\{ -(\sin\vartheta+3\mathrm i\cos\vartheta)e^{3\mathrm i\vartheta}\right\}  \bigg \vert_{\vartheta_m}^{\vartheta_M}+\mathcal O(\tau^{-1}e^{-\frac{1}{2}\lambda r_0\tau})\label{2eq:L3}.
\end{align}
The second term of \eqref{2eq:d2q} can be written as:
\begin{align}
\frac{\Gamma(3)\overline {q}_2e^{-3\mathrm i\varphi}}{\tau^3}&\int_{\vartheta_m}^{\vartheta_M}\sin\vartheta e^{3\mathrm i\vartheta}\mathrm d\vartheta
+\mathcal O(\tau^{-1}e^{-\frac{1}{2}\varsigma r_0\tau})\notag\\
&=\frac{\Gamma(3)e^{-3\mathrm i\varphi}}{\tau^3}\cdot \frac{1}{8}\left\{ (\cos\vartheta-3\mathrm i\sin\vartheta)e^{3\mathrm i\vartheta}\right\}  \bigg \vert_{\vartheta_m}^{\vartheta_M}+\mathcal O(\tau^{-1}e^{-\frac{1}{2}\lambda r_0\tau})\label{2eq:L4}.
\end{align}

Combining \eqref{2eq:este1}-\eqref{2eq:L4} with \eqref{2eq:q1q2}, one has
\begin{align}
\bigg \vert&\frac{\Gamma(3)e^{3\mathrm i\varphi}}{8\tau^3}\left\{ \overline q_1 \left ((-\sin \vartheta+3\mathrm i\cos\vartheta)e^{3 \mathrm i\vartheta} \right)+\overline q_2 \left( (\cos\vartheta+3\mathrm i\sin\vartheta)e^{-3\mathrm i\vartheta}\right)  \right\}\bigg \vert_{\vartheta_m}^{\vartheta_M} \bigg \vert\label{2eq:TL1}\\
&\leq \tau^{-\frac{41}{12}}+\tau^{-(3+\alpha)}+(1+\tau)(1+\tau^{2/3})e^{-\varsigma r_0\tau}
+\mathcal O(\tau^{-1}e^{-\frac{1}{2}\varsigma r_0\tau})+\mathcal O(\tau^{-1}e^{-\lambda r_0\tau})\notag
\end{align}
and
\begin{align}
\bigg \vert &\frac{\Gamma(3)e^{-3\mathrm i\varphi}}{8\tau^3}\left\{ \overline q_1 \left (-(\sin\vartheta+3\mathrm i\cos\vartheta)e^{3\mathrm i\vartheta} \right)+\overline q_2 \left( (\cos\vartheta-3\mathrm i\sin\vartheta)e^{3\mathrm i\vartheta}\right)  \right\}\bigg \vert_{\vartheta_m}^{\vartheta_M}\bigg \vert \label{2eq:TL2} \\
&\leq\tau^{-\frac{41}{12}}+\tau^{-(3+\alpha)}+(1+\tau)(1+\tau^{2/3})e^{-\varsigma r_0\tau}
+\mathcal O(\tau^{-1}e^{-\frac{1}{2}\varsigma r_0\tau})+\mathcal O(\tau^{-1}e^{-\lambda r_0\tau})\notag.
\end{align}

By multiplying both sides of \eqref{2eq:TL1} by $\frac{8\tau^3}{\Gamma(3)e^{3\mathrm i\varphi}}$, \eqref{2eq:TL2} by $\frac{8\tau^3e^{3\mathrm i\varphi}}{\Gamma(3)}$, and letting $\tau\to \infty$, we obtain
\begin{equation}\notag
B \left(\begin{matrix}
\overline q_1\\
\overline q_2
\end{matrix}\right)=\mathbf 0,
\end{equation}
where
\begin{equation}\notag
B=
\left(\begin{matrix}
(-\sin\vartheta+3\mathrm i\cos\vartheta)e^{-3\mathrm i\vartheta}\vert_{\vartheta_m}^{\vartheta_M} &
(\cos\vartheta+3\mathrm i\sin\vartheta)e^{-3\mathrm i\vartheta}\vert_{\vartheta_m}^{\vartheta_M} \\
-(\sin\vartheta+3\mathrm i\cos\vartheta)e^{3\mathrm i\vartheta}\vert_{\vartheta_m}^{\vartheta_M} &
(\cos\vartheta-3\mathrm i\sin\vartheta)e^{3\mathrm i\vartheta}\vert_{\vartheta_m}^{\vartheta_M}
\end{matrix}\right).
\end{equation}
Furthermore, it can be calculated that
\begin{align}
\det(B)&=(20\sin(3\beta)\sin\beta+12\cos(3\beta)\cos\beta-12)\mathrm i\notag\\
&=-4(\cos(2\beta)-1)^2\mathrm i\not=0,\  \beta=\vartheta_M-\vartheta_m\in(0,\pi).\notag
\end{align}
Hence we get that  $\overline q_1=0$ and $\overline q_2=0$, which implies that $q_1^{(1)}=q_2^{(1)}$ and $q_1^{(2)}=q_2^{(2)}$ by using $ u_2^{-}(\mathbf 0)\not =0.$

\medskip

\bigskip

Finally, we shall prove $u_1^{-}=u_2^{-}  $ in $\mathcal S_{r_0}$. By the fact that  $\eta_1=\eta_2$ and $q_1=q_2$ and the boundary condition, then we have
\begin{equation}\notag
\begin{cases}
\Delta w+k^2 q w=0 \hspace*{0.4cm}\mbox{in}\ \mathcal S_{r_0},\\
w=\partial_\nu  w=0 \hspace*{0.9cm}\mbox{on}\ \Gamma_{r_0}^\pm,
\end{cases}
\end{equation}
where $w=u_2^{-}-u_1^{-}.$ Using Holmgren's  principle (cf.\cite{CK}), we have $w=0 $ in $\mathcal S_{r_0}$.

The proof is complete.
\end{proof}

\section{The proofs of Theorems \ref{2thm:un-eta-cell} and \ref{2thm:un-eta-nest}}\label{sec:4}

\begin{proof}[Proof of Theorem \ref{2thm:un-eta-cell}]

In accordance with Theorem \ref{2thm:un-eta-cell},  the shape of the conductive medium scatterer $\Omega$ and the corresponding polygonal-cell structures $\Sigma_i,i=1,\dots,N$, are known in advance. We establish uniqueness results for determining the physical parameters of the refractive index and the conductive boundary parameter through a proof by contradiction. Assume that $\tilde\eta_1\not=\tilde \eta_2$ or there is an index $i_0$ such that $q_{i_0,1}\not=q_{i_0,2}$. From Definition \ref{1-def:cell}, we know that there exists a vertex $\mathbf x_{p_0}\in \Sigma_{i_0}$ intersected by two adjacent edges $\Gamma^{\pm}$ of $\Omega$. Since $-\Delta$ is invariant under rigid motion, without loss of generality, we assume that $\mathbf x_{p_0}$ coincides with the origin. Let $r_0\in\mathbb R_{+}$ be sufficiently small such that $\mathcal S_{r_0}=\Omega\cap B_{r_0}(\mathbf 0)$ and $\Gamma_{r_0}^\pm=\partial \Omega\cap B_{r_0}(\mathbf 0)$. Denote $W=\mathbb R^{2}\setminus\overline{\Omega},$ then we have $\Gamma_{r_0}^\pm\subset\partial  W$, see Figure \ref{2fig:cell2} for a schematic illustration.
\begin{figure}[htpb]
\centering
\includegraphics[width=7cm]{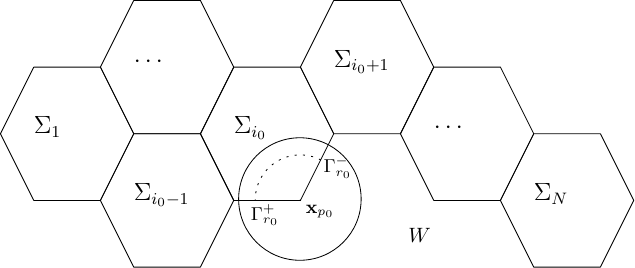}
\caption{Schematic illustration of the polygonal-cell structure.}
\label{2fig:cell2}
\end{figure}

Recall that $u_1$ and $u_2$ are the total wave fields associated with the admissible conductive  polygonal-cell medium scatterers $(\Omega;q_\ell,\eta_\ell),\ell=1,2$.  Since $u_1^{\infty}(\mathbf {\hat{x}};u^i)=u_2^\infty(\mathbf {\hat x};u^i)$ for all $\mathbf {\hat{x}}\in \mathbb S^{1},$  by Rellich lemma and the unique continuation principle, we have
\begin{equation}\notag
u_1=u_2\ \mbox{in}\ B_{r_0}\setminus \overline{\mathcal S_{r_0}}\subset W.
\end{equation}
Utilizing the conductive transmission boundary conditions in \eqref{eq:sca}, one can claim that
\begin{equation}\notag
\begin{cases}
u_1^{+}=u_1^{-},\ \partial_\nu u_1^{-}=\partial_\nu u_1^{+}+\tilde \eta_1 u_1^{+},\ \mbox{on}\ \Gamma_{r_0}^{\pm},\\
u_2^{+}=u_2^{-},\ \partial_\nu u_2^{-}=\partial_\nu u_2^{+}+\tilde \eta_2 u_2^{+},\ \mbox{on}\ \Gamma_{r_0}^{\pm}.
\end{cases}
\end{equation}
Moreover, it holds that
\begin{equation}\notag
\begin{cases}
\Delta u_1^{-}+k^2q_{i_0,1}u_1^{-}=0,\ \hspace*{0.5cm}\mbox{in}\ \mathcal S_{r_0},\\
\Delta u_1^{+}+k^2u_1^{+}=0,  \hspace*{1.3cm}\mbox{in}\ B_{r_0}\setminus \overline{\mathcal S_{r_0}},\\
\Delta u_2^{-}+k^2q_{i_0,2}u_2^{-}=0,\ \hspace*{0.5cm}\mbox{in}\ \mathcal S_{r_0},\\
\Delta u_2^{+}+k^2u_2^{+}=0,  \hspace*{1.3cm}\mbox{in}\ B_{r_0}\setminus \overline{\mathcal S_{r_0}}.
\end{cases}
\end{equation}
Recall that $\Omega$ is an admissible conductive polygonal-cell medium scatterer. Then by virtue of Lemma \ref{2thm:conductive}, we have
$$q_{i_0,1}=q_{i_0,2}\ \mbox{and}\ \tilde \eta_1=\tilde \eta_2,$$
which is a contradiction.

The proof is complete.
\end{proof}

\begin{proof}[The proof of Theorem \ref{2thm:un-eta-nest}.]
We prove this theorem by mathematical induction. By Theorem \ref{2thm:un-shape}, we have $\partial \Omega_1=\partial \Omega_2$, which indicates that $\partial \Sigma_{1,1}=\partial \Sigma_{1,2}.$ Taking a similar argument as proving Theorem \ref{2thm:un-eta-cell}, it can be shown that $q_{1,1}=q_{1,2}$ and $\eta_{1,1}=\eta_{1,2}$. Suppose that there exists an index $n_0\in \mathbb N_{+}\setminus\{1\}$ such that
$$\partial \Sigma_{i,1}=\partial \Sigma_{i,2},\ q_{i,1}=q_{i,2},\ \eta_{i,1}=\eta_{i,2},\ i=1,\dots,n_0-1.$$
With the help of  Lemma \ref{2thm:conductive}, we can recursively obtain that
\begin{equation}\label{2eq:ul}
u_{i,1}=u_{i,2}\ \mbox{in}\ \mathcal D_{i}=\Sigma_{i}\setminus\overline{\Sigma_{i+1}},\ i=1,\dots,n_0-1
\end{equation}
by making use of $u_1^\infty(\mathbf {\hat x};u^i)=u_2^\infty(\mathbf {\hat{x}};u^i)$, where $u_{i,1}=u_1|_{\mathcal D_i}$ and $u_{i,2}=u_2|_{\mathcal D_i}$, with $u_1$ and $u_2$ being the total wave fields associated with the admissible conductive polygonal-nest medium scatterers $(\Omega_\ell;q_\ell,\eta_\ell),\ell=1,2.$ We divide the rest proof into two parts.

\medskip
 \noindent\textbf{Part 1}. We first prove that $\partial \Sigma_{n_0,1}=\partial \Sigma_{n_0,2}$.  According to \eqref{2eq:ul},  using the  similar argument in proving Theorem \ref{2thm:un-shape}, one can prove  $\partial \Sigma_{n_0,1}=\partial \Sigma_{n_0,2}$ directly by contradiction.

\begin{figure}[htpb]
\centering
\includegraphics[width=6cm]{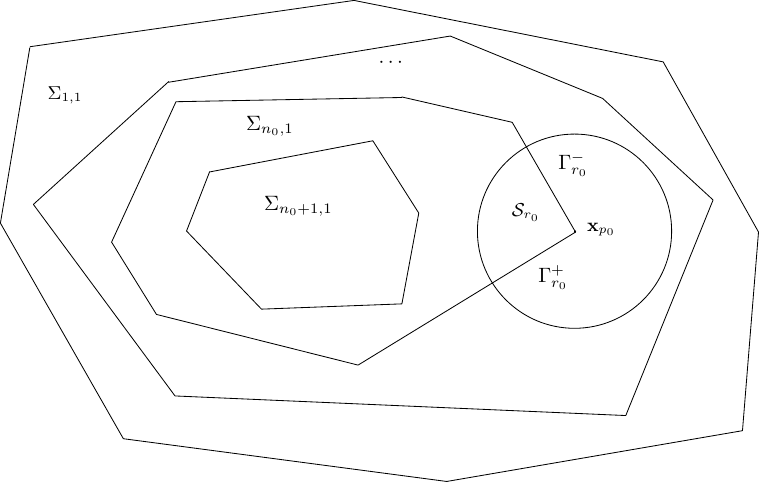}\\
\caption{Schematic illustration of the polygonal-nest structure.}
\label{2fig:nest2}
\end{figure}
\medskip
 \noindent\textbf{Part 2}. In the following we  prove that
\begin{equation}\label{2eq:qeta}
q_{n_0,1}=q_{n_0,2},\ \eta_{n_0,1}=\eta_{n_0,2}.
\end{equation}
Let $\mathbf x_{p_0}$ be a vertex of $\Sigma_{n_0,1}$. Without loss of generality, we may assume that $\mathbf x_{p_0}=\mathbf 0$ since $-\Delta$ is invariant under rigid motion. For sufficiently small $r_0>0$, assume that $\mathcal S_{r_0}\Subset \mathcal D_{n_0,1}=\Sigma_{n_0,1}\setminus\overline{\Sigma_{n_0+1,1}}$; refer to Figure \ref{2fig:nest2} for a schematic illustration.
Therefore, it yields that
\begin{equation}\label{2eq:unest}
\begin{cases}
\Delta u_{n_0,1}^{-}+k^2q_{n_0,1}u_{n_0,1}^{-}=0\hspace*{4.95cm}\mbox{in}\ \mathcal S_{r_0},\\
\Delta u_{n_0-1,1}^{+}+k^2q_{n_0-1,1}u_{n_0-1,1}^{+}=0\hspace*{3.8cm}\mbox{in}\ B_{r_0}\setminus \overline{\mathcal S_{r_0}},\\
\Delta u_{n_0,2}^{-}+k^2q_{n_0,2}u_{n_0,2}^{-}=0\hspace*{4.95cm}\mbox{in}\ \mathcal S_{r_0},\\
\Delta u_{n_0-1,2}^{+}+k^2q_{n_0-1,2}u_{n_0-1,2}^{+}=0\hspace*{3.8cm}\mbox{in}\ B_{r_0}\setminus \overline{\mathcal S_{r_0}},\\
u_{n_0,1}^{-}=u_{n_0-1,1}^{+},\ \partial_\nu u_{n_0,1}^{-}=\partial_\nu u_{n_0-1,1}^{+}+\eta_{n_0,1} u_{n_0-1,1}^{+}, \hspace*{0.5cm}\mbox{on}\ \Gamma_{r_0}^\pm,\\
u_{n_0,2}^{-}=u_{n_0-1,2}^{+},\ \partial_\nu u_{n_0,2}^{-}=\partial_\nu u_{n_0-1,2}^{+}+\eta_{n_0,2} u_{n_0-1,2}^{+}, \hspace*{0.5cm}\mbox{on}\ \Gamma_{r_0}^\pm.\\
\end{cases}
\end{equation}
It is obvious  that $u_{n_0-1,1}^{+}=u_{n_0-1,2}^{+}$ by virtue of \eqref{2eq:ul}. In view of \eqref{2eq:unest}, by Lemma \ref{2thm:conductive}, one has \eqref{2eq:qeta}.

Finally, we can prove $N_1=N_2$   through a proof by contradiction. In fact, if $N_1\not=N_2$, we assume, without loss of generality, that $N_1>N_2$. Therefore, there  exists a corner $\Sigma_{N_1+1,1}$ with a vertex $\mathbf x_{p_0}$ inside $\Sigma_{N_2,2}$. Using a similar argument in proving Theorem \ref{2thm:un-shape}, we can prove that either the total wave field or the gradient of the total wave field vanishes at the vertex of this corner, which contradicts to the  admissible condition of a conductive polygonal-nest medium scatterer in Definition \ref{2def:adm}.

The proof is complete.
\end{proof}

\section*{Acknowledgements}
The work of H. Diao is supported by by National Natural Science Foundation of China  (No. 12371422) and the Fundamental Research Funds for the Central Universities, JLU. The work of X. Fei is supported by NSFC/RGC Joint Research Grant No. 12161160314. The work of H. Liu was supported by the Hong Kong RGC General Research Funds (projects 11311122, 11300821 and 11304224), NSF/RGC Joint Research Fund (project N\_CityU101/21) and the ANR/RGC Joint Research Fund (project A\_CityU203/19).


\begin{thebibliography}{99}

\bibitem{AD}
{\sc R.~A. Adams and J.~J. Fournier}, {\em Sobolev spaces}, vol.~140, Elsevier,
  2003.

\bibitem{AR}
{\sc G.~Alessandrini and L.~Rondi}, {\em Determining a sound-soft polyhedral
  scatterer by a single far-field measurement}, Proc. Amer. Math. Soc., 133
  (2005), pp.~1685--1691, {https://doi.org/10.1090/S0002-9939-05-07810-X}.

\bibitem{ARRV}
{\sc G.~Alessandrini, L.~Rondi, E.~Rosset, and S.~Vessella}, {\em The stability
  for the {C}auchy problem for elliptic equations}, Inverse Problems, 25
  (2009), pp.~123004, 47, {https://doi.org/10.1088/0266-5611/25/12/123004}.

\bibitem{ak92}
{\sc T.~S. Angell and A.~Kirsch}, {\em The conductive boundary condition for
  {M}axwell's equations}, SIAM J. Appl. Math., 52 (1992), pp.~1597--1610,
  {https://doi.org/10.1137/0152092}.

\bibitem{AS}
{\sc S.~N. Armstrong and L.~Silvestre}, {\em Unique continuation for fully
  nonlinear elliptic equations}, Math. Res. Lett., 18 (2011), pp.~921--926,
 {https://doi.org/10.4310/MRL.2011.v18.n5.a9}.

\bibitem{BL}
{\sc E.~Bl\aa~sten and H.~Liu}, {\em Recovering piecewise constant refractive
  indices by a single far-field pattern}, Inverse Problems, 36 (2020),
  pp.~085005, 16, {https://doi.org/10.1088/1361-6420/ab958f}.

\bibitem{BL2021}
{\sc E.~Bl\aa~sten and H.~Liu}, {\em Scattering by curvatures, radiationless
  sources, transmission eigenfunctions, and inverse scattering problems}, SIAM
  J. Math. Anal., 53 (2021), pp.~3801--3837,
 {https://doi.org/10.1137/20M1384002}.

\bibitem{BPS2014}
{\sc E.~Bl\aa~sten, L.~P\"aiv\"arinta, and J.~Sylvester}, {\em Corners always
  scatter}, Comm. Math. Phys., 331 (2014), pp.~725--753,
  {https://doi.org/10.1007/s00220-014-2030-0}.

\bibitem{Bon}
{\sc O.~Bondarenko}, {\em The factorization method for conducting transmission
  conditions}, PhD thesis, Dissertation, Karlsruhe, Karlsruher Institut f{\"u}r
  Technologie (KIT), 2016.

\bibitem{bhk}
{\sc O.~Bondarenko, I.~Harris, and A.~Kleefeld}, {\em The interior transmission
  eigenvalue problem for an inhomogeneous media with a conductive boundary},
  Appl. Anal., 96 (2017), pp.~2--22,
  {https://doi.org/10.1080/00036811.2016.1204440}.

\bibitem{Bon-Liu}
{\sc O.~Bondarenko and X.~Liu}, {\em The factorization method for inverse
  obstacle scattering with conductive boundary condition}, Inverse Problems, 29
  (2013), pp.~095021, 25, {https://doi.org/10.1088/0266-5611/29/9/095021}.

\bibitem{CV23}
{\sc F.~Cakoni and M.~S. Vogelius}, {\em Singularities almost always scatter:
  regularity results for non-scattering inhomogeneities}, Comm. Pure Appl.
  Math., 76 (2023), pp.~4022--4047, {https://doi.org/10.1002/cpa.22117}.

\bibitem{CVX23}
{\sc F.~Cakoni, M.~S. Vogelius, and J.~Xiao}, {\em On the regularity of
  non-scattering anisotropic inhomogeneities}, Arch. Ration. Mech. Anal., 247
  (2023), pp.~Paper No. 31, 15,
  {https://doi.org/10.1007/s00205-023-01863-y}.

\bibitem{CX}
{\sc F.~Cakoni and J.~Xiao}, {\em On corner scattering for operators of
  divergence form and applications to inverse scattering}, Comm. Partial
  Differential Equations, 46 (2021), pp.~413--441,
  {https://doi.org/10.1080/03605302.2020.1843489}.

\bibitem{CDL20}
{\sc X.~Cao, H.~Diao and H.~Liu}, {\em Determining a piecewise conductive
  medium body by a single far-field measurement}, CSIAM Transactions on Applied
  Mathematics, 1 (2020), pp.~740--765,
  {https://doi.org/https://doi.org/10.4208/csiam-am.2020-0020}.

\bibitem{CDLZ20}
{\sc X.~Cao, H.~Diao, H.~Liu, and J.~Zou}, {\em On nodal and generalized
  singular structures of {L}aplacian eigenfunctions and applications to inverse
  scattering problems}, J. Math. Pures Appl. (9), 143 (2020), pp.~116--161,
  {https://doi.org/10.1016/j.matpur.2020.09.011}.

\bibitem{Carleman}
{\sc T.~Carleman}, {\em Sur un probl\`eme d'unicit\'e{} pur les syst\`emes
  d'\'equations aux d\'eriv\'ees partielles \`a{} deux variables
  ind\'ependantes}, vol.~26, 1939.

\bibitem{CY03}
{\sc J.~Cheng and M.~Yamamoto}, {\em Uniqueness in an inverse scattering
  problem within non-trapping polygonal obstacles with at most two incoming
  waves}, Inverse Problems, 19 (2003), pp.~1361--1384,
  {https://doi.org/10.1088/0266-5611/19/6/008}.

\bibitem{CK}
{\sc D.~Colton and R.~Kress}, {\em Inverse acoustic and electromagnetic
  scattering theory}, vol.~93 of Applied Mathematical Sciences, Springer, New
  York, third~ed., 2013, {https://doi.org/10.1007/978-1-4614-4942-3}.

\bibitem{CK18}
{\sc D.~Colton and R.~Kress}, {\em Looking back on inverse scattering theory},
  SIAM Rev., 60 (2018), pp.~779--807, {https://doi.org/10.1137/17M1144763}.

\bibitem{CS83}
{\sc D.~Colton and B.~D. Sleeman}, {\em Uniqueness theorems for the inverse
  problem of acoustic scattering}, IMA J. Appl. Math., 31 (1983), pp.~253--259,
  {https://doi.org/10.1093/imamat/31.3.253}.

\bibitem{DN}
{\sc M.~Dauge and S.~Nicaise}, {\em Oblique derivative and interface problems
  on polygonal domains and networks}, Comm. Partial Differential Equations, 14
  (1989), pp.~1147--1192, {https://doi.org/10.1080/03605308908820649}.

\bibitem{DCL21}
{\sc H.~Diao, X.~Cao, and H.~Liu}, {\em On the geometric structures of
  transmission eigenfunctions with a conductive boundary condition and
  applications}, Comm. Partial Differential Equations, 46 (2021), pp.~630--679,
  {https://doi.org/10.1080/03605302.2020.1857397}.

\bibitem{DFL}
{\sc H.~Diao, X.~Fei, and H.~Liu}, {\em Local geometric properties of
  conductive transmission eigenfunctions and applications}, European Journal of
  Applied Mathematics,  (2024), p.~1šC32,
  {https://doi.org/10.1017/S0956792524000287}.

\bibitem{Diao25ip}
{\sc H.~Diao, X.~Fei, H.~Liu, and L.~Wang}, {\em Determining anomalies in a
  semilinear elliptic equation by a minimal number of measurements}, Inverse
  Problems,  (2025), {https://doi.org/10.1088/1361-6420/adc82a}.

\bibitem{Diaofei24ipi}
{\sc H.~Diao, X.~Fei, H.~Liu, and K.~Yang}, {\em Visibility, invisibility and
  unique recovery of inverse electromagnetic problems with conical
  singularities}, Inverse Probl. Imaging, 18 (2024), pp.~541--570,
  {https://doi.org/10.3934/ipi.2023043}.

\bibitem{Diaotang23jde}
{\sc H.~Diao, H.~Li, H.~Liu, and J.~Tang}, {\em Spectral properties of an
  acoustic-elastic transmission eigenvalue problem with applications}, J.
  Differential Equations, 371 (2023), pp.~629--659,
{https://doi.org/10.1016/j.jde.2023.07.002}.

\bibitem{Diao23book}
{\sc H.~Diao and H.~Liu}, {\em Spectral geometry and inverse scattering
  theory}, Springer, Cham, 2023,
  {https://doi.org/10.1007/978-3-031-34615-6}.

\bibitem{Diao25effective}
{\sc H.~Diao, H.~Liu, Q.~Meng, and L.~Wang}, {\em Effective medium theory for
  embedded obstacles in electromagnetic scattering with applications}, J.
  Differential Equations, 437 (2025), p.~113283,
  {https://doi.org/10.1016/j.jde.2025.113283}.

\bibitem{Diaotao24ip}
{\sc H.~Diao, H.~Liu, and L.~Tao}, {\em Stable determination of an impedance
  obstacle by a single far-field measurement}, Inverse Problems, 40 (2024),
  pp.~Paper No. 055005, 35, {https://doi.org/10.1088/1361-6420/ad3087}.

\bibitem{Diaowang22jde}
{\sc H.~Diao, H.~Liu, and L.~Wang}, {\em Further results on generalized
  {H}olmgren's principle to the {L}am\'e{} operator and applications}, J.
  Differential Equations, 309 (2022), pp.~841--882,
  {https://doi.org/10.1016/j.jde.2021.11.039}.

\bibitem{Diaozhang21ip}
{\sc H.~Diao, H.~Liu, L.~Zhang, and J.~Zou}, {\em Unique continuation from a
  generalized impedance edge-corner for {M}axwell's system and applications to
  inverse problems}, Inverse Problems, 37 (2021), pp.~Paper No. 035004, 32,
  {https://doi.org/10.1088/1361-6420/abdb42}.

\bibitem{Diaotang24ipi}
{\sc H.~Diao, R.~Tang, H.~Liu, and J.~Tang}, {\em Unique determination by a
  single far-field measurement for an inverse elastic problem}, Inverse Probl.
  Imaging, 18 (2024), pp.~1405--1430,
  {https://doi.org/10.3934/ipi.2024020}.

\bibitem{EH18}
{\sc J.~Elschner and G.~Hu}, {\em Acoustic scattering from corners, edges and
  circular cones}, Arch. Ration. Mech. Anal., 228 (2018), pp.~653--690,
  {https://doi.org/10.1007/s00205-017-1202-4}.

\bibitem{GL1}
{\sc N.~Garofalo and F.-H. Lin}, {\em Monotonicity properties of variational
  integrals, {$A_p$} weights and unique continuation}, Indiana Univ. Math. J.,
  35 (1986), pp.~245--268, {https://doi.org/10.1512/iumj.1986.35.35015}.

\bibitem{GL}
{\sc N.~Garofalo and F.-H. Lin}, {\em Unique continuation for elliptic
  operators: a geometric-variational approach}, Comm. Pure Appl. Math., 40
  (1987), pp.~347--366, {https://doi.org/10.1002/cpa.3160400305}.

\bibitem{GRSU}
{\sc T.~Ghosh, A.~R\"uland, M.~Salo, and G.~Uhlmann}, {\em Uniqueness and
  reconstruction for the fractional {C}alder\'on problem with a single
  measurement}, J. Funct. Anal., 279 (2020), pp.~108505, 42,
  {https://doi.org/10.1016/j.jfa.2020.108505}.

\bibitem{GT-ell}
{\sc D.~Gilbarg and N.~S. Trudinger}, {\em Elliptic partial differential
  equations of second order}, vol.~Vol. 224 of Grundlehren der Mathematischen
  Wissenschaften, Springer-Verlag, Berlin-New York, 1977.

\bibitem{hk20}
{\sc I.~Harris and A.~Kleefeld}, {\em The inverse scattering problem for a
  conductive boundary condition and transmission eigenvalues}, Appl. Anal., 99
  (2020), pp.~508--529, {https://doi.org/10.1080/00036811.2018.1504028}.

\bibitem{HLNS}
{\sc N.~Honda, C.-L. Lin, G.~Nakamura, and S.~Sasayama}, {\em Unique
  continuation property of solutions to general second order elliptic systems},
  J. Inverse Ill-Posed Probl., 30 (2022), pp.~5--21,
  {https://doi.org/10.1515/jiip-2020-0073}.

\bibitem{Horm1}
{\sc L.~H\"ormander}, {\em The analysis of linear partial differential
  operators. {III}}, Classics in Mathematics, Springer, Berlin, 2007,
  {https://doi.org/10.1007/978-3-540-49938-1}.
\newblock Pseudo-differential operators, Reprint of the 1994 edition.

\bibitem{Horm2}
{\sc L.~H\"ormander}, {\em The analysis of linear partial differential
  operators. {IV}}, Classics in Mathematics, Springer-Verlag, Berlin, 2009,
  {https://doi.org/10.1007/978-3-642-00136-9}.
\newblock Fourier integral operators, Reprint of the 1994 edition.

\bibitem{HSV16}
{\sc G.~Hu, M.~Salo, and E.~V. Vesalainen}, {\em Shape identification in
  inverse medium scattering problems with a single far-field pattern}, SIAM J.
  Math. Anal., 48 (2016), pp.~152--165,
  {https://doi.org/10.1137/15M1032958}.

\bibitem{JK85}
{\sc D.~Jerison and C.~E. Kenig}, {\em Unique continuation and absence of
  positive eigenvalues for {S}chr\"odinger operators}, Ann. of Math. (2), 121
  (1985), pp.~463--494, {https://doi.org/10.2307/1971205}.
\newblock With an appendix by E. M. Stein.

\bibitem{KT01}
{\sc H.~Koch and D.~Tataru}, {\em Carleman estimates and unique continuation
  for second-order elliptic equations with nonsmooth coefficients}, Comm. Pure
  Appl. Math., 54 (2001), pp.~339--360,
  {https://doi.org/10.1002/1097-0312(200103)54:3<339::AID-CPA3>3.0.CO;2-D}.

\bibitem{KSS24}
{\sc P.-Z. Kow, M.~Salo, and H.~Shahgholian}, {\em On scattering behavior of
  corner domains with anisotropic inhomogeneities}, SIAM J. Math. Anal., 56
  (2024), pp.~4834--4853, {https://doi.org/10.1137/23M1603029}.

\bibitem{LLM}
{\sc M.~Lassas, T.~Liimatainen, and M.~Salo}, {\em The {C}alder\'on problem for
  the conformal {L}aplacian}, Comm. Anal. Geom., 30 (2022), pp.~1121--1184,
 {https://doi.org/10.4310/cag.2022.v30.n5.a6}.

\bibitem{LP}
{\sc P.~D. Lax and R.~S. Phillips}, {\em Scattering theory}, vol.~26 of Pure
  and Applied Mathematics, Academic Press, New York-London, 1967.

\bibitem{Ler}
{\sc N.~Lerner}, {\em Carleman inequalities}, vol.~353 of Grundlehren der
  mathematischen Wissenschaften [Fundamental Principles of Mathematical
  Sciences], Springer, Cham, 2019,
  {https://doi.org/10.1007/978-3-030-15993-1}.
\newblock An introduction and more.

\bibitem{LZ1}
{\sc H.~Liu and J.~Zou}, {\em Uniqueness in an inverse acoustic obstacle
  scattering problem for both sound-hard and sound-soft polyhedral scatterers},
  Inverse Problems, 22 (2006), pp.~515--524,
  {https://doi.org/10.1088/0266-5611/22/2/008}.

\bibitem{MZ92}
{\sc Z.~Mghazli}, {\em Regularity of an elliptic problem with mixed
  {D}irichlet-{R}obin boundary conditions in a polygonal domain}, Calcolo, 29
  (1992), pp.~241--267, {https://doi.org/10.1007/BF02576184}.

\bibitem{SNS}
{\sc E.~Moreira~dos Santos, G.~Nornberg, and N.~Soave}, {\em On unique
  continuation principles for some elliptic systems}, Ann. Inst. H.
  Poincar\'e{} C Anal. Non Lin\'eaire, 38 (2021), pp.~1667--1680,
  {https://doi.org/10.1016/j.anihpc.2020.12.001}.

\bibitem{PSV}
{\sc L.~P\"aiv\"arinta, M.~Salo, and E.~V. Vesalainen}, {\em Strictly convex
  corners scatter}, Rev. Mat. Iberoam., 33 (2017), pp.~1369--1396,
  {https://doi.org/10.4171/RMI/975}.

\bibitem{Protter}
{\sc M.~H. Protter}, {\em Unique continuation for elliptic equations}, Trans.
  Amer. Math. Soc., 95 (1960), pp.~81--91,
  {https://doi.org/10.2307/1993331}.

\bibitem{ST}
{\sc A.~Salda\~na and H.~Tavares}, {\em Least energy nodal solutions of
  {H}amiltonian elliptic systems with {N}eumann boundary conditions}, J.
  Differential Equations, 265 (2018), pp.~6127--6165,
  {https://doi.org/10.1016/j.jde.2018.07.013}.

\bibitem{PS21}
{\sc M.~Salo and H.~Shahgholian}, {\em Free boundary methods and non-scattering
  phenomena}, Res. Math. Sci., 8 (2021), pp.~Paper No. 58, 19,
  {https://doi.org/10.1007/s40687-021-00294-z}.

\bibitem{Sogge}
{\sc C.~D. Sogge}, {\em Fourier integrals in classical analysis}, vol.~210 of
  Cambridge Tracts in Mathematics, Cambridge University Press, Cambridge,
  second~ed., 2017, {https://doi.org/10.1017/9781316341186}.

\bibitem{Tat}
{\sc D.~Tataru}, {\em Unique continuation problems for partial differential
  equations}, in Geometric methods in inverse problems and {PDE} control,
  vol.~137 of IMA Vol. Math. Appl., Springer, New York, 2004, pp.~239--255,
  {https://doi.org/10.1007/978-1-4684-9375-7\_8}.

\end{thebibliography}
	\end{document}